\documentclass{amsart}

\usepackage{tikz,amsthm,amsmath,amstext,amssymb,epsfig,euscript, mathrsfs, dsfont,pspicture,multicol,graphpap,graphics,graphicx,times,enumerate,subfig,sidecap,wrapfig,color}
\usepackage{enumerate}

 \numberwithin{equation}{section}

\def\bC{{\mathbb{C}}}

\def\bR{{\mathbb{R}}}

\def\bZ{{\mathbb{Z}}}

\def\bN{{\mathbb{N}}}

\def\cB{{\mathscr{B}}}

\def\cH{{\mathscr{H}}}

\def\cL{{\mathscr{L}}}

\def\one{\mathds{1}}

\def\ve{\varepsilon}

\renewcommand{\d}{{\partial}}
\def\lec{\lesssim}
\def\gec{\gtrsim}

\def\diam{\text{diam}}
\def\dim{\mbox{dim}}
\def\dist{\mbox{dist}}

\newcommand{\ps}[1]{\left( #1 \right)}

\newcommand{\ck}[1]{\left\{#1 \right\}}

\newcommand{\ceil}[1]{\left\lceil #1 \right\rceil}
\newcommand{\cnj}[1]{\overline{#1}}

\def\barintgerm_#1{\mathchoice
{\mathop{\vrule width 6pt height 3 pt depth -2.5pt
\kern -8.8pt \intop}\nolimits_{#1}}%
{\mathop{\vrule width 5pt height 3 pt depth -2.6pt
\kern -6.5pt \intop}\nolimits_{#1}}%
{\mathop{\vrule width 5pt height 3 pt depth -2.6pt
\kern -6pt \intop}\nolimits_{#1}}%
{\mathop{\vrule width 5pt height 3 pt depth -2.6pt \kern -6pt
\intop}\nolimits_{#1}}}

\theoremstyle{plain}
\newtheorem{theorem}{Theorem}

\newtheorem{lemma}[theorem]{Lemma}
\newtheorem{sublemma}[theorem]{Sublemma}
\newtheorem{proposition}[theorem]{Proposition}

\theoremstyle{definition}

\newtheorem{definition}[theorem]{Definition}
\newtheorem{remark}[theorem]{Remark}

\newcommand\eqn[1]{\eqref{e:#1}}
\def\Claim{ {\bf Claim: }}
\newcommand\Theorem[1]{Theorem \ref{t:#1}}
\newcommand\Lemma[1]{Lemma \ref{l:#1}}

\newcommand\Proposition[1]{Proposition \ref{p:#1}}
\begin{document}

\def\bhat{\hat{\beta}}

\def\putgrid{\put(0,0){0}
\put(0,25){25}
\put(0,50){50}
\put(0,75){75}
\put(0,100){100}
\put(0,125){125}
\put(0,150){150}
\put(0,175){175}
\put(0,200){200}
\put(25,0){25}
\put(50,0){50}
\put(75,0){75}
\put(100,0){100}
\put(125,0){125}
\put(150,0){150}
\put(175,0){175}
\put(200,0){200}
\put(225,0){225}
\put(250,0){250}
\put(275,0){275}
\put(300,0){300}
\put(325,0){325}
\put(350,0){350}
\put(375,0){375}
\put(400,0){400}
}

\newcommand{\Section}[1]{Section \ref{s:#1}}

\title{Hausdorff dimension of wiggly metric spaces}
\author{Jonas Azzam%\footnote{Supported by RTG DMS 08-38212}}
}

\address{J. Azzam\\
University of Washington
          Department of Mathematics \\
         C--138 Padelford Hall \\
         Seattle, WA, 98195-4350}
\email {jazzam@math.washington.edu}
\thanks{The author was supported by the NSF grants RTG DMS 08-38212 and DMS-0856687.
}

%28A75,  Length, area, volume, other geometric measure theory
%28A78, Hausdorff and packing measures
%30L10. quasiconformal mappings
\maketitle
\maketitle
\begin{abstract} 
For a compact connected set $X\subseteq \ell^{\infty}$, we define a quantity $\beta'(x,r)$ that measures how close $X$ may be approximated in a ball $B(x,r)$ by a geodesic curve. We then show there is $c>0$ so that if $\beta'(x,r)>\beta>0$ for all $x\in X$ and $r<r_{0}$, then $\dim X>1+c\beta^{2}$. This generalizes a theorem of Bishop and Jones and answers a question posed by Bishop and Tyson.\\

%quasiconformal mappings
\end{abstract}

\section{Introduction}
\subsection{Background and Main Results}

Our starting point is a theorem of Bishop and Jones, stated below, which roughly says that a connected subset of $\bR^{2}$ that is uniformly non-flat in every ball centered upon it  (or in other words, is very ``wiggly"), must have large dimension. We measure flatness with Jones' $\beta$-numbers: if $K$ is a subset of a Hilbert space $\cH$, $x\in K$ and $r>0$, we define

\begin{equation}
\beta(x,r)=\beta_{K}(x,r)=\frac{1}{r}\inf_{L}\sup\{\dist(y,L):y\in K\cap B(x,r)\}
\label{e:euclidean-beta}
\end{equation}
where the infimum is taken over all lines $L\subseteq \cH$. 

\begin{theorem}(\cite[Theorem 1.1]{BJ91-wiggly}) There is a constant $c>0$ such that the following holds. Let $K\subseteq \bR^{2}$ be a compact connected set and suppose that there is $r_{0}>0$ such that for all $r\in (0,r_{0})$ and all $x\in K$, $\beta_{K}(x,r)>\beta_{0}$. Then the Hausdorff dimension\footnote{See \Section{prelims} for the definition of Hausdorff dimension and other definitions and notation.} of $K$ satisfies $\dim K\geq 1+c\beta_{0}^{2}$. 
\label{t:BJ}
\end{theorem}

There are also analogues of \Theorem{BJ} for surfaces of higher topological dimension, see for example \cite{Guy04}.  

Our main theorem extends this result to the metric space setting using an alternate definition of $\beta$. Before stating our results, however, we discuss the techniques and steps involved in proving \Theorem{BJ} to elucidate why the original methods don't immediately carry over, and to discuss how they must be altered for the metric space setting.\\

The main tool in proving \Theorem{BJ} is the {\it Analyst's Traveling Salesman Theorem}, which we state below. First recall that for a metric space $(X,d)$, a {\it maximal $\ve$-net}  is a maximal collection of points $X'\subseteq X$ such that $d(x,y)\geq \ve$ for all $x,y\in X'$. 

\begin{theorem}(\cite[Theorem 1.1]{Schul-TSP}) Let $A>1$, $K$ be a compact subset of a Hilbert space $\cH$, and $X_{n}\supseteq X_{n+1}$ be a nested sequence of maximal $2^{-n}$-nets in $K$. For $A>1$, define
\begin{equation}
\beta_{A}(K):=\diam K+\sum_{n\in\bZ}\sum_{x\in X_{n}} \beta_K^{2}(x,A2^{-n} )2^{-n}.
\label{e:betaK}
\end{equation}
There is $A_{0}$ such that for $A>A_{0}$ there is $C_{A}>0$ (depending only on $A$) so that for any $K$, $\beta_{A}(K)<\infty$ implies there is a connected set $\Gamma$ such that $K\subseteq \Gamma$ and
\[\cH^{1}(\Gamma)\leq C_{A} \beta_{A}(K).\]
Conversely, if $\Gamma$ is connected and $\cH^{1}(\Gamma)<\infty$, then for any $A>1$,
\begin{equation}
\beta_{A}(\Gamma)\leq C_{A} \cH^{1}(\Gamma).
\label{e:beta_gamma}
\end{equation}
\label{t:TST}
\end{theorem}

At the time of \cite{BJ91-wiggly}, this was only known for the case $\cH=\bR^2$, due to Jones \cite{Jones-TSP}. This was subsequently generalized to $\bR^{n}$ by Okikiolu \cite{O-TSP} and then to Hilbert space by Schul \cite{Schul-TSP}.

The proof of \Theorem{BJ} goes roughly as follows: one constructs a  {\it Frostmann measure} $\mu$ supported on $K$ satisfying \begin{equation}
\mu(B(x,r))\leq C r^{s}
\label{e:frostmann}
\end{equation}
for some $C>0$, $s=1+c \beta_{0}^{2}$ and for all $x\in K$ and $r>0$. This easily implies that the Hausdorff dimension of $K$ is at least $s$ (see \cite[Theorem 8.8]{Mattila} and that section for a discussion on Frostmann measures). One builds such a measure on $K$ inductively by deciding the values $\frac{\mu(Q_n)}{\mu(Q)}$ for each dyadic cube $Q$ intersecting $K$ and for each $n$-th generation descendant $Q_n$ intersecting $K$, where $n$ is some large number that will depend on $\beta_{0}$. If the number of such $n$-th generation descendants is large enough, we can choose the ratios and hence disseminate the mass $\mu(Q)$ amongst the descendants $Q_{n}$ in such a way that the ratios will be very small and \eqn{frostmann} will be satisfied. To show that there are enough descendants, one looks at the skeletons of the $n$-th generation descendants of $Q$ and uses the second half of \Theorem{TST} coupled with the non-flatness condition in the satement of \Theorem{BJ} to guarantee that the total length of this skeleton (and hence the number of cubes) will be large. 

In the metric space setting, however, no such complete analogue of \Theorem{TST} exists, and it is not even clear what the appropriate analogue of a $\beta$-number should be. Note, for example, that it does not make sense to estimate the length of a metric curve $\Gamma$ using the original $\beta$-number, even if we consider $\Gamma$ as lying in some Banach space. A simple counter example is if $\Gamma\subseteq L^{1}([0,1])$ is the image of $s:[0,1]\rightarrow L^{1}([0,1])$ defined by $t\mapsto \one_{[0,t]}$. This a geodesic, so in particular, it is a rectifiable curve of finite length. However, $\beta_{\Gamma}(x,r)$ (i.e. the width of the smallest tube containing $\Gamma\cap B(x,r)$ in $L^{1}$, rescaled by a factor $r$) is uniformly bounded away from zero, and in particular, $\beta_{A}(\Gamma)=\infty$. 

In \cite{Hah05}, Hahlomaa gives a good candidate for a $\beta$-number for a general metric space $X$ using Menger curvature and uses it to show that if the sum in \eqn{betaK} is finite for $K=X$ (using his definition of $\beta_{X}$), then it can be contained in the Lipschitz image of a subset of the real line (analogous to the first half of \Theorem{TST}). An example of Schul \cite{Schul-survey}, however, shows that the converse of \Theorem{TST} is false in general: \eqn{beta_gamma} with Hahlomaa's $\beta_{X}$ does not hold with the same constant for all curves in $\ell^{1}$. We refer to \cite{Schul-survey} for a good summary on the Analyst's Traveling Salesman Problem. 

To generalize \Theorem{BJ}, we use a $\beta$-type quantity that differs from both Jones' and Hahlomaa's definitions. It is inspired by one defined by Bishop and Tyson in \cite{BT01-antenna} that measures the deviation of a set from a geodesic in a metric space: if $X$ is a metric space, $B_{X}(x,r)=\{y\in X:d(x,y)<r)\}$, and $y_{0},...,y_{n}\in B_{X}(x,r)$ an ordered sequence, define
\begin{equation}
\d(y_{0},...,y_{n})=\sum_{i=0}^{n-1}d(y_{i},y_{i+1}) -d(y_{0},y_{n}) +\sup_{z\in B_{X}(x,r)}\min_{i=1,...,n}d(z,y_{i})
\label{e:d}
\end{equation}
and define
\begin{equation}
\hat{\beta}_{X}(x,r)= \inf_{\{y_{i}\}\subseteq B_{X}(x,r)} \frac{\d(y_{0},...,y_{n})}{d(y_{0},y_{n})}
\label{e:bd}
\end{equation}
where the infimum is over all finite ordered sequences in $B_{X}(x,r)$ of any length $n$. 

In \cite{BT01-antenna}, Bishop and Tyson ask whether, for a compact connected metric space $X$, \eqn{bd} being uniformly larger than zero is enough to guarantee that $\dim X>1$. We answer this in the affirmative. 

\begin{theorem}
There is $\kappa>0$ such that the following holds. If $X$ is a compact connected metric space and $\hat{\beta}_{X}(x,r)>\beta>0$ for all $x\in X$ and $r\in(0,r_{0})$ for some $r_{0}>0$, then  $\dim X\geq 1+\kappa\beta^{4}$. 
\label{t:BT-answer}
\end{theorem}

Instead of $\bhat$, however, we work with a different quantity, which we define here for a general compact metric space $X$. First, by Kuratowski embedding theorem, we may assume $X$ is a subset of $\ell^{\infty}$, whose norm we denote by $|\cdot |$. Let $B(x,r)=B_{\ell^{\infty}}(x,r)$ and define
\begin{equation}
\beta_{X}'(x,r) =\inf_{s} \frac{\ell(s)-|s(0)-s(1)| + \sup_{z\in X\cap B(x,r)}\dist (z,s([0,1]))}{|s(0)-s(1)|}
\end{equation}
where the infimum is over all curves $s:[0,1]\rightarrow B(x,r)\subseteq \ell^{\infty}$ and
\[\ell(s)= \sup_{\{t_{i}\}_{i=0}^{n}} \sum_{i=0}^{n-1} |s(t_{i})-s(t_{i+1})|\]
is the length of $s$, where the supremum is over all partitions $0=t_{0}<t_{1}<\cdots <t_{n}=1$. In general, if $s$ is defined on a union of disjoint open intervals $\{I_{j}\}_{j=1}^{\infty}$, we set
\[\ell(s|_{\bigcup I_{j}})=\sum_{j} \ell(s|_{I_{j}}).\]
The case in which $s$ is just a straight line segment through the center of the ball with length $2r$ gives the estimate $\beta_{X}'(x,r)\leq \frac{1}{2}$.

The quantity $\beta'(x,r)$ measures how well $X\cap B(x,r)$ may be approximated by a geodesic. To see this, note that if, for some $s:[0,1]\rightarrow\ell^{\infty}$, the $\frac{\beta'(x,r)}{2}|s(0)-s(1)|$-neighborhood of $s([0,1])$ contains $X\cap B(x,r)$, then the length of $s$ must be at least $(1+\frac{\beta'(x,r)}{2})|s(0)-s(1)|$, which is $\frac{\beta'(x,r)}{2}|s(0)-s(1)|$ more than the length of any geodesic connecting $s(0)$ and $s(1)$. The quantity $\hat{\beta}$ similarly measures how well the portion of $X\cap B(x,r)$ may be approximated by a geodesic polygonal path with vertices in $X$. In Figure \ref{f:betas}, we compare the meanings of $\beta,\bhat,$ and $\beta'$. 

We will refer to the quantities $\ell(s)$ and $\d(y_{0},...,y_{n})$ as the {\it geodesic deviation} of $s$ and $\{y_{0},...,y_{n}\}$ respectively. We will also say  $\hat{\beta}_{X}(x,r)$ and $\beta_{X}'(x,r)$ measure the {\it geodesic deviation} of $X$ inside the ball $B(x,r)$. 

%\begin{figure}[t]
%\begin{picture}(100,250)(0,0)
%\put(0,0){\includegraphics[width=390pt]{betas.pdf}} %[width=0.8\textwidth]
%%\putgrid
%\put(55,215){ $\beta(x,r)2r$}
%\put(45,110){${ B(y_{i},\beta|y_{0}-y_{n}|)}$}
%\put(85,20){${ |y_{0}-y_{n}|}$}
%\put(250,15){${ |s(0)-s(1)|}$}
%%\put(230,90){${ d(B\cap \Gamma,s([0,1]))}$}
%\put(235,93){$<\beta|s(0)-s(1)|$}
%\put(300,38){${ s([0,1])}$}
%\put(275,200){${B= B(x,r)}$}
%\put(90,165){$\Gamma$}
%\end{picture}
%\caption{ In each of the three figures above is a ball $B=B(x,r)$ containing a portion of a curve $\Gamma$. In the first picture, $\beta(x,r)2r$ is the width of the smallest tube containing $ X\cap B(x,r)$. In the second, we see that $\hat{\beta}(x,r)$ is such that for $\beta>\hat{\beta}(x,r)$, there are $y_{0},...,y_{n}\in \Gamma$ with vertices in $ X\cap B$ so that the $\beta|y_{0}-y_{n}|$ neighborhood of the vertices contain $ X\cap B$, and so that the geodesic deviation (that is, its length minus $|y_{0}-y_{n}|$ is at most $\beta|y_{0}-y_{n}|$. In the last, we show that if $\beta'(x,r)<\beta$, there is $s:[0,1]\rightarrow \ell^{\infty}$ whose geodesic deviation and whose distance from any point in $ X\cap B$ are at most $\beta|s(0)-s(1)|$.}
%\label{f:betas}
%\end{figure}

\begin{figure}[t!]
\begin{picture}(100,300)(130,0)
\put(0,0){\includegraphics[width=360pt]{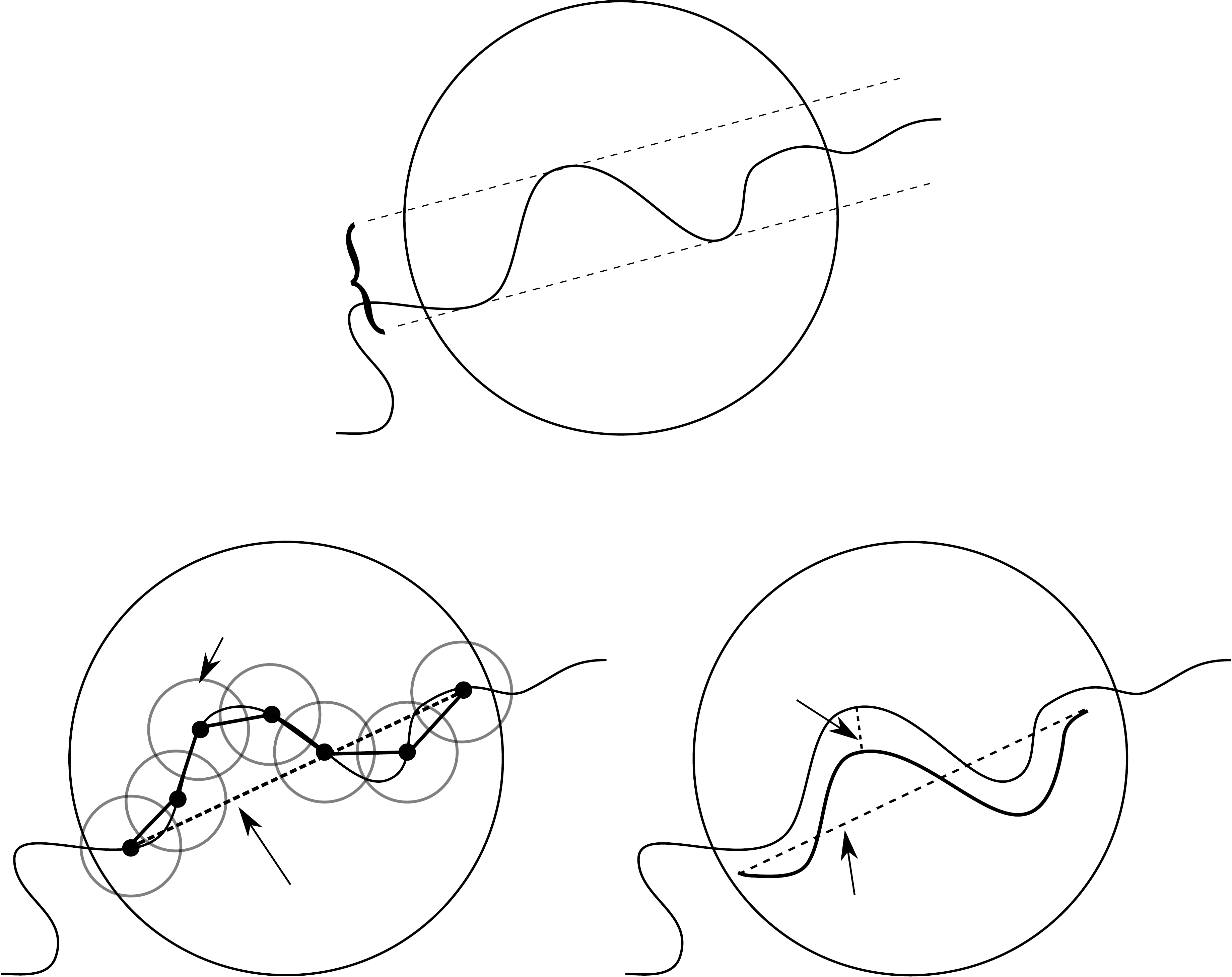}} %[width=0.8\textwidth]
%\putgrid
\put(55,200){ $\beta(x,r)2r$}
\put(45,105){${ B(y_{i},\beta|y_{0}-y_{n}|)}$}
\put(85,20){${ |y_{0}-y_{n}|}$}
\put(240,15){${ |s(0)-s(1)|}$}
%\put(230,90){${ d(B\cap X,s([0,1]))}$}
\put(220,85){$<\beta|s(0)-s(1)|$}
\put(285,38){${ s([0,1])}$}
\put(275,200){${B= B(x,r)}$}
\put(90,165){$X$}
\end{picture}
\caption{ In each of the three figures above is a ball $B=B(x,r)$ containing a portion of a curve $X$. In the first picture, $\beta(x,r)2r$ is the width of the smallest tube containing $X\cap B(x,r)$. In the second, we see that $\hat{\beta}(x,r)$ is such that for $\beta>\hat{\beta}(x,r)$, there are $y_{0},...,y_{n}\in X$ with vertices in $X\cap B$ so that balls centered on the $y_{i}$ of radius $\beta|y_{0}-y_{n}|$ cover $X\cap B$, and so that the geodesic deviation (that is, its length minus $|y_{0}-y_{n}|$ is at most $\beta|y_{0}-y_{n}|$. In the last, we show that if $\beta'(x,r)<\beta$, there is $s:[0,1]\rightarrow \ell^{\infty}$ whose geodesic deviation and whose distance from any point in $X\cap B$ are both at most $\beta|s(0)-s(1)|$.}
\label{f:betas}
\end{figure}
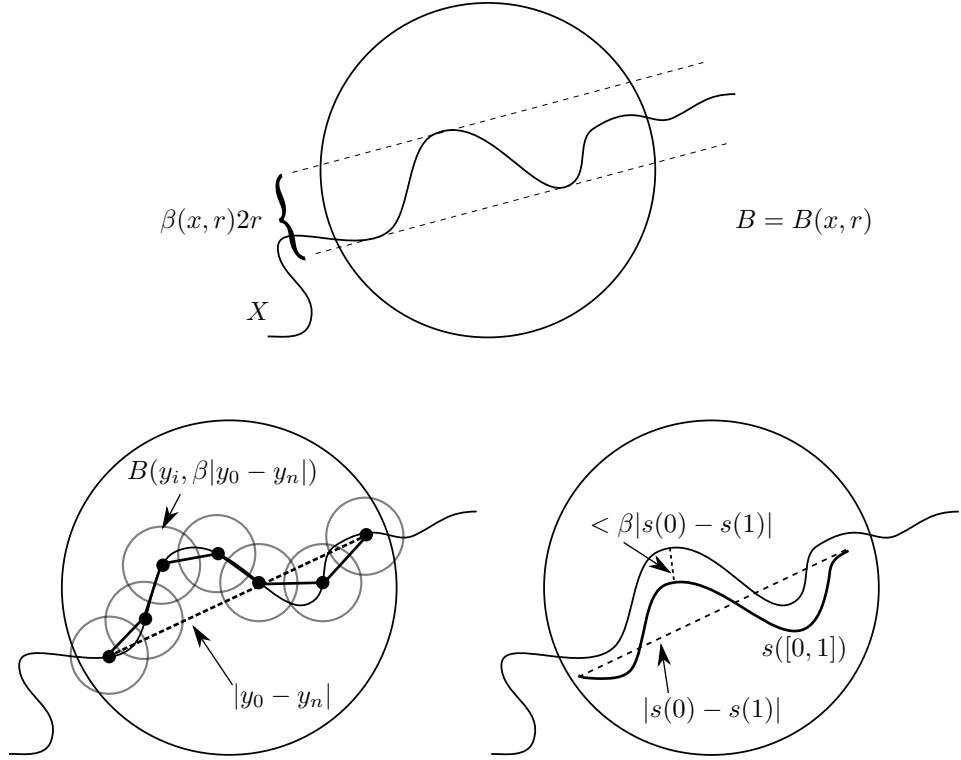

Note that for the image of $t\mapsto\one_{[0,t]}\in L^{1}([0,1])$ described earlier, it is easy to check that $\bhat(x,r)=\beta'(x,r)=0$ for all $x\in X$ and $r>0$, even though $\beta_{X}(x,r)$ is bounded away from zero. This, of course, makes the terminology ``wiggly" rather misleading in metric spaces, since there are certainly non-flat or highly ``wiggly" geodesics in $L^{1}$;  we use this terminology only to be consistent with the literature. Later on in \Proposition{bb''}, however, we will show that in a Hilbert space we have  for some $C>0$,
\begin{equation}
\beta'(x,r)\leq \beta(x,r)  \leq C \beta'(x,r)^{\frac{1}{2}}.
\label{e:bb'-intro}
\end{equation}
That the two should be correlated in this setting seems natural as $\beta(x,r)$ is measuring how far $X$ is deviating from a straight line, which are the only geodesics in Hilbert space.

In \Lemma{beta'-bhat} below, we will also show that for some $C>0$,
\[
\beta'(x,r)\leq \bhat(x,r)  \leq C \beta'(x,r)^{\frac{1}{2}}
\]
so that \Theorem{BT-answer} follows from the following theorem, which is our main result.

\begin{theorem}
There is $c_{0}>0$ such that the following holds. If $X$ is a compact connected metric space and $\beta'_{X}(x,r)>\beta>0$ for all $x\in X$ and $r\in(0,r_{0})$ for some $r_{0}>0$, then $\dim X\geq 1+c_{0}\beta^{2}$. 
\label{t:main}
\end{theorem}
We warn the reader, however, that the quadratic dependence on $\beta$ appears in \Theorem{main} and \Theorem{BJ} for completely different reasons. In \Theorem{BJ}, it comes from using \Theorem{TST}, or ultimately from the Pythagorean theorem, which of course does no hold in general metric spaces; in \Theorem{main}, it seems to be an artifact of the construction and can perhaps be improved.

Our approach to proving \Theorem{main} follows the original proof of \Theorem{BJ} described earlier: to show that a metric curve $ X$ has large dimension, we approximate it by a polygonal curve, estimate its length from below and use this estimate to construct a Frostmann measure, but in lieu of a traveling salesman theorem. (In fact, taking $\beta'(x,A2^{-n})$ instead of $\beta(x,A2^{-n})^2$ in \Theorem{TST} does not lead to a metric version of \Theorem{TST} for a similar reason that Hahlomaa's $\beta$-number doesn't work; one need only consider Schul's example \cite[Section 3.3.1]{Schul-survey}.)\\

\subsection{An Application to Conformal Dimension}

The original context of Bishop and Tyson's conjecture, and the motivation for \Theorem{main}, concerned conformal dimension. Recall that a {\it quasisymmetric map} $f:X\rightarrow Y$ between two metric spaces is a map for which there is an increasing homeomorphism $\eta:(0,\infty)\rightarrow(0,\infty)$ such that for any distinct $x,y,z\in X$, 
\[\frac{|f(x)-f(y)|}{|f(z)-f(y)|}\leq \eta\ps{\frac{|x-y|}{|z-y|}}.\]
The {\it conformal dimension} of a metric space $X$ is 
\[\mbox{C-dim} X=\inf_{f}\dim f(X)\]
where the infimum ranges over all quasisymmetric maps $f:X\rightarrow f(X)$. For more information, references, and recent work on conformal dimension, see for example \cite{conformal-dimension}.

 In \cite{BT01-antenna}, it is shown that the antenna set has conformal dimension one yet every quasisymmetric image of it into any metric space has dimension strictly larger than one. The {\it antenna set} is a self similar fractal lying in $\bC$ whose similarities are the following: 
\[f_{1}(z)=\frac{z}{2},\;\; f_{2}(z)=\frac{z+1}{2}, \;\; f_{3}(z)=i\alpha z+\frac{1}{2},f_{4}(z)=-i\alpha z+\frac{1}{2}+i\alpha\]
where $\alpha\in (0,\frac{1}{2})$ is some fixed angle (see Figure \ref{f:antenna}).

%\begin{figure}[h]
%\includegraphics[width=\textwidth]{antenna.pdf}
%
%\caption{The antenna set with $\alpha=\frac{1}{4}$.}
%\label{f:antenna}
%\end{figure}

\begin{figure}[h]
\includegraphics[width=\textwidth]{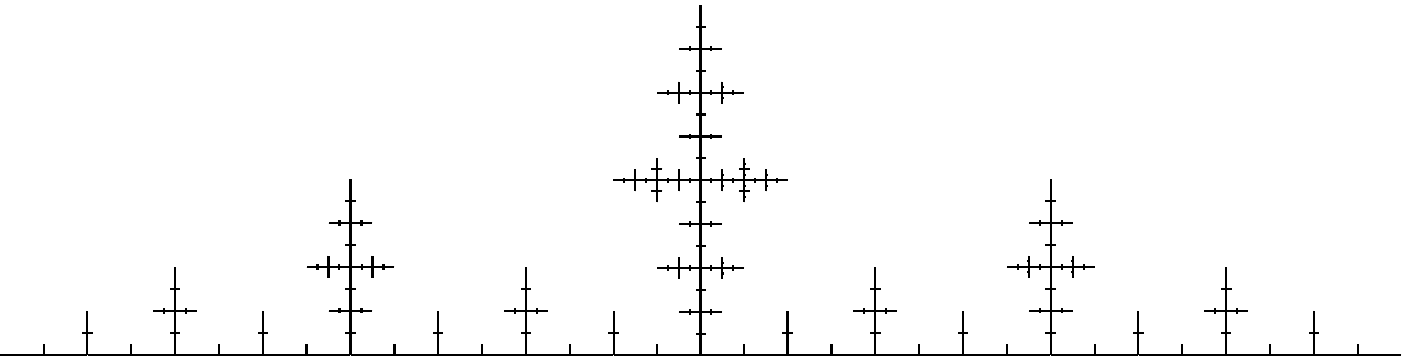}

\caption{The antenna set with $\alpha=\frac{1}{4}$.}
\label{f:antenna}
\end{figure}

To show the conformal dimension $1$ is never attained under any quasisymmetric image of the antenna set, the authors show by hand that any quasisymmetic map of the antenna set naturally induces a Frostmann measure of dimension larger than one. At the end of the paper, however, the authors suggested another way of showing the same result by proving an analogue of \Theorem{BJ} for a $\beta$-number which is uniformly large for the antenna set as well as any quasisymmetric image of it. 

\Theorem{main} doesn't just give a much longer proof of Bishop and Tyson's result, but it lends itself to more general sets lacking any self-similar structure.

\begin{definition}
Let $c>0$, $Y=[0,e_{1}]\cup [0,e_{2}]\cup [0,e_{3}]\subseteq \bR^{3}$, where $e_{j}$ is the $j$th standard basis vector in $\bR^{3}$, and let $X$ be a compact connected metric space. For $x\in X$, $r>0$, we say $B_{X}(x,r)$ has a {\it $c$-antenna} if there is a homeomorphism $h:Y\rightarrow h(Y)\subseteq B_{X}(x,r)$ such that the distance between $h(e_{i})$ and $h([0,e_{j}]\cup [0,e_{k}]))$ is at least $cr$ for all permutations $(i,j,k)$ of $(1,2,3)$. We say $X$ is {\it $c$-antenna-like} if $B_{X}(x,r)$ has a $c$-antenna for every $x\in X$ and $r<\frac{\diam X}{2}$,
\end{definition}

%
%\begin{figure}[h]
%\begin{picture}(100,200)(-70,0)
%\scalebox{.7}{\includegraphics{antenna-like.pdf}}
%\put(-25,30){$x$}
%\put(-25,70){$B(x,r)$}
%\put(0,0){\vector(-1,1){17}}
%\put(0,-5){$h(Y)$}
%\put(-175,150){$X$}
%\end{picture}
%\caption{An example of an antenna-like set.}
%\label{f:antenna-like}
%\end{figure}

Clearly, the classical antenna set in $\bR^{2}$ is antenna-like.

\begin{theorem}
Let $ X$ be a compact connected metric space in $\ell^{\infty}$.
\begin{enumerate}
\item If $B_{X}(x,r)$ has a $c$-antenna, then $\beta'(x,r)>\frac{c}{7}$. Hence, if $ X$ is $c$-antenna-like, we have $\dim  X\geq 1+\frac{c_{0}}{49} c^{2}$.
\item Any quasisymmetric image of an antenna-like set into any metric space is also antenna-like and hence has dimension strictly larger than one.
\end{enumerate}
\label{t:antenna-like}
\end{theorem}

Note that this result doesn't say the conformal dimension of an antenna-like set is larger than one, only that no quasisymmetric image of it has dimension equal to one. However, see \cite{Mackay10}, where the author bounds the conformal dimension of a set from below using a different quantity. 

%
%In \cite{TW06}, Tyson and Wu show that the {\it quasiconformal dimension} of the Sierpinski Gasket $SG\subseteq \bR^{2}$ (as well as other polygasket fractals) is one. This means that the infimum of $\dim f(SG)$ over all quasiconformal homeomorphisms $f:\bR^{2}\rightarrow \bR^{2}$ is one. \Theorem{antenna-like} shows in addition that this infimum is never attained even if one allows quasisymmetric images into arbitrary metric spaces.
%
%\begin{corollary}
%If $K\subseteq \bR^{d}$ is a path-wise connected  self-similar fractal not contained in a line, then $\dim f(K)>1$ for all $f:K\rightarrow f(K)$ quasisymmetric. In particular, $\dim f(SG)>1$ for all quasisymmetric maps $f:SG\rightarrow f(SG)$.
%\end{corollary}
%
%The proof is simple: if $K$ satisfies the conditions of the theorem, then it must be $c$-antenna-like with some constant $c$.

\subsection{Outline}
In \Section{prelims}, we go over some necessary notation and tools before proceeding to the proof of \Theorem{main} in \Section{proof}. In \Section{antenna}, we prove \Theorem{antenna-like}, and in \Section{betas} we compare $\beta',\bhat,$ and $\beta$.

\subsection{Acknowledgements}
The author would like to thank Steffen Rohde, Tatiana Toro, and Jeremy Tyson for their helpful discussions, and to Matthew Badger, John Garnett, Raanan Schul, and the anonymous referee for their helpful comments on the manuscript. Part of this manuscript was written while the author was at the IPAM long program Interactions Between
Analysis and Geometry, Spring 2013.

\section{Preliminaries}
\label{s:prelims}

\subsection{Basic notation}
Since we are only dealing with compact metric spaces, by the Kuratowski embedding theorem, we will implicitly assume that all our metric spaces are contained in $\ell^{\infty}$, whose norm we will denote $|\cdot|$.

For $x\in \ell^{\infty}$ and $r>0$, we will write
\[B(x,r)=\{y\in\ell^{\infty}:|x-y|<r\}\subseteq\ell^{\infty}.\] 
If $B=B(x,r)$ and $\lambda>0$, we write $\lambda B$ for $B(x,\lambda r)$.
For a set $A\subseteq \ell^{\infty}$ and $\delta>0$, define
\[A_{\delta}=\{x\in \ell^{\infty}:\dist(x,A)<\delta\} \;\; \mbox{ and } \;\;\diam A=\sup\{|x-y|:x,y\in A\}\]
where 
\[\dist(A,B)=\inf\{|x-y|: x\in A,y\in B\}, \;\;\; \dist(x,A)=\dist(\{x\},A).\]
For a set $E\subseteq \bR$, let $|E|$ denote its Lebesgue measure. For an interval $I\subseteq \bR$, we will write $a_{I}$ and $b_{I}$ for its left and right endpoints respectively. For $s>0$, $\delta\in (0,\infty]$ and $A\subseteq \ell^{\infty}$, define
\[\cH_{\delta}^{s}(A)=\inf\ck{ \sum\diam A_{j}: A\subseteq \bigcup A_{j}, \diam A_{j}<\delta},\]
\[\cH^{s}(A)=\lim_{\delta\rightarrow 0} \cH_{\delta}^{1}(A).\]
The {\it Hausdorff dimension} of a set $A$ is 
\[\dim A:=\inf\{s:\cH^{s}(A)=0\}.\]

\subsection{Cubes}

In this section, we construct a family of subsets of $\ell^{\infty}$, tailored to a metric space $X$, that have properties similar to dyadic cubes in Euclidean space. These cubes appeared in \cite{Schul-TSP} (where they were alternatively called  ``cores") and are similar to the so-called Christ-David Cubes (\cite{David88,Christ-T(b)}) in some respects, although they are not derived from them.

Fix $M>0$ and $c\in (0,\frac{1}{8})$. Let  $X_{n}\subseteq X$ be a nested sequence of maximal $M^{-n}$-nets in $X$. Let
\[\cB_{n}=\{B(x,M^{-n}): x\in X_{n}\}, \;\; \cB=\bigcup_{n} \cB_{n}.\]
For $B=B(x,M^{-n})\in \cB_{n}$, define
\[Q_{B}^{0}=cB, \;\; Q_{B}^{j}=Q_{B}^{j-1}\cup\bigcup\{cB: B\in \bigcup_{m\geq n} \cB_{m}, cB\cap Q_{B}^{j-1}\neq\emptyset\}, Q_{B}=\bigcup_{j=0}^{\infty} Q_{B}^{j}.\]
Basically, $Q_{B}$ is the union of all balls $B'$ that may be connected to $B$ by a chain $\{cB_{j}\}$ with $B_{j}\in \cB$, $\diam B_{j}\leq \diam B$, and $cB_{j}\cap cB_{j+1}$ for all $j$. 

For such a cube $Q$ constructed from $B(x,M^{-n})$, we let $x_{Q}=x$ and $B_{Q}=B(x,cM^{-n})$. 

Let 
\[\Delta_{n}=\{Q_{B}:B\in \cB_{n}\},  \;\; \Delta=\bigcup \Delta_{n}.\] 
Note that, for $Q\in \Delta_{n}$, $x_{Q}\in X_{n}$.

\begin{lemma}
If $c<\frac{1}{8}$, then for $X$ and $\Delta$ as above, the family of cubes $\Delta$ satisfy the following properties. 
\begin{enumerate}
\item If $Q,R\in \Delta$ and $Q\cap R\neq\emptyset$, then $Q\subseteq R$ or $R\subseteq Q$.
\item For $Q\in \Delta$, 
\begin{equation}
B_{Q}\subseteq Q\subseteq (1+8M^{-1})B_{Q}.
\label{e:1+2N^-1}
\end{equation}
\end{enumerate}
\label{l:cubes}
\end{lemma}

The proof is essentially in \cite{Schul07}, but with slightly different parameters. So that the reader need not perform the needed modifications, we provide a proof here.

\begin{proof}
Part 1 follows from the definition of the cubes $Q$. To prove Part 2, we first claim that if $\{B_{j}\}_{j=0}^{n}$ is a chain of balls with centers $x_{j}$ for which $cB_{j}\cap cB_{j+1}\neq\emptyset$, then for $C=\frac{1}{1-2M^{-1}}$,
\begin{equation}
\sum_{j=0}^{n} \diam cB_{j} \leq C \max_{j=0,...,n} \diam cB_{j}.
\label{e:ballchain}
\end{equation}
We prove \eqn{ballchain} by induction. Let $x_{j}$ denote the center of $B_{j}$ If $n=1$, $\diam B_{0}\leq \diam B_{1}$, and $x_{0}$ and $x_{1}$ are the centers of $B_{0}$ and $B_{1}$ respectively, then $\diam B_{0}\leq M^{-1}\diam B_{1}$ since otherwise $B_{0},B_{1}\in \cB_{N}$ for some $N$ and
\[ M^{-n}\leq |x_{0}-x_{1}|\leq \frac{\diam cB_{0}}{2} + \frac{\diam cB_{1}}{2}= 2cM^{-n}<M^{-n}\]
since $c<\frac{1}{8}$, which is a contradiction. Hence,
\[ \diam cB_{0}+\diam cB_{1}\leq (1+2M^{-1}) \diam cB_{1} \leq C \diam cB_{1}.\]
Now suppose $n>1$. Let $j_{0}\in \{1,...,n\}$ and $N$ be an integer so that 
\begin{equation}
\diam B_{j_{0}}=\max_{j=1,...,n} \diam B_{j}=2M^{-N}.
\label{e:maxball}
\end{equation}
Recall that all balls in $\cB$ have radii that are powers of $M^{-1}$, so there exists an $N$ so that the above happens. 

 Note that $B_{j_{0}-1}$ and $B_{j_{0}}$ cannot have the same diameter (which follows from the $n=1$ case we proved earlier). Since $B_{j_{0}}$ has the maximum diameter of all the $B_{j}$, we in fact know that $\diam B_{j_{0}-1}\leq M^{-1}B_{j_{0}}$ (again, recall that all balls have radii that are powers of $M^{-1}$). 
 
 Let $i_{0}\leq j_{0}$ be the minimal integer for which $\diam B_{i_{0}}\leq M^{-1} \diam B_{j_{0}}$ (which exists by the previous discussion) and let $k_{0}\geq j_{0}$ be the maximal integer such that $B_{k_{0}}\leq M^{-1} \diam B_{j_{0}}$. By the induction hypothesis,
\[\sum_{j=j_{0}+1}^{k_{0}} \diam cB_{j}\leq C \max_{j_{0}<j\leq k_{0}}\diam cB_{j}\leq CM^{-1}  \diam cB_{j_{0}}\]
and 
\begin{equation}
\sum_{j=i_{0}}^{j_{0}-1} \diam cB_{j}\leq C\max_{i_{0}\leq j<j_{0}}\diam cB_{j}\leq CM^{-1} \diam cB_{j_{0}}
\label{e:itoj-1}
\end{equation}
so that
\begin{equation}
\sum_{j=i_{0}}^{k_{0}}\diam B_{j} \leq (1+2CM^{-1})\diam cB_{j_{0}}=C\diam c B_{j_{0}}.
\label{e:itok}
\end{equation}
\Claim $i_{0}=0$. Note that if $i_{0}>0$, then
\begin{align*}
|x_{i_{0}-1}-x_{j_{0}}|
& \leq \sum_{i=i_{0}-1}^{j_{0}}\diam cB_{i}
 \leq \diam c B_{i_{0}-1}+\diam c B_{j_{0}} + \sum_{i=i_{0}}^{j_{0}-1} 2 c B_{j_{0}} \\
&  \stackrel{\eqn{maxball} \atop \eqn{itoj-1}}{\leq}2\diam c B_{j_{0}} + CM^{-1}\diam  c B_{j_{0}}\\
&  = (2c+cCM^{-1})\diam B_{j_{0}}=(2c+cCM^{-1})2M^{-N}
 <M^{-N}
\end{align*}
for $c<\frac{1}{4}$ and $M>4$ (this makes $C<2$). Since $x_{j_{0}}\in X_{N}$ and points in $X_{N}$ are $M^{-N}$-separated, we must have $x_{i_{0}-1}\not\in X_{N}$, hence $B_{i_{0}-1}\not\in \cB_{N}$. Thus, 
\[\diam B_{i_{0}-1}\leq M^{-1}\diam B_{j_{0}},\] 
which contradicts the minimality of $i_{0}$, hence $i_{0}=0$. We can prove similarly that $k_{0}=n$, and this with \eqn{itoj-1} proves \eqn{ballchain}. This in turn implies that for any $N\in\bN$, if $Q\in \Delta_{N}$, then $\diam Q\leq C\diam cB_{Q}$, hence
\begin{align*}
Q
& \subseteq B(x_{Q},cM^{-N}+(C-1)\diam cB_{Q})
 = B\ps{x_{Q}, c\ps{1+\frac{4M^{-1}}{1-2M^{-1}}}M^{-N}}\\
& \subseteq (1+8M^{-1})B_{Q}.
\end{align*}

\end{proof}

For $N$ large enough, this means we can pick our cubes so that they don't differ much from balls. We will set $8M^{-1}=\ve\beta$ for some $\ve\in (0,1)$ to be determined later, so that
\begin{equation}
B_{Q}\subseteq Q\subseteq (1+\ve\beta)B_{Q}
\label{e:1+veb}
\end{equation}

\begin{remark}
There are a few different constructions of families of metric subsets with properties similar to dyadic cubes, see \cite{David88}, \cite{Christ-T(b)}, and \cite{HK12} for example, and the references therein. Readers familiar with any of these references will see that Schul's ``cores"  we have just constructed are very different from the cubes constructed in the aforementioned references. In particular, each $\Delta_{n}$ does not partition any metric space in the same way that dyadic cubes (half-open or otherwise) would partition Euclidean space, not even up to set of measure zero). However,  for each $n$ we do have
 \begin{equation}
  X\subseteq \bigcup \{ c^{-1}Q:Q\in \Delta_{n}\},
  \label{e:1/cQ}
  \end{equation}
 and we still have the familiar intersection properties in \Lemma{cubes}. The reason for the ad hoc construction is the crucial ``roundness" property \eqn{1+veb}. 
\end{remark}

\begin{lemma}
Let $\gamma:[0,1]\rightarrow \ell^{\infty}$ be a piecewise linear function and set $\Gamma=\gamma([0,1])$, whose image is a finite union of line segments, and let $\Delta$ be the cubes from \Lemma{cubes} tailored to $X$. Then for any $Q\in \Delta$, $\cH^{1}(\d Q)=0$ and $|\gamma^{-1}(\d Q)|=0$. 
\label{l:zero-boundary}
\end{lemma}

\begin{proof}
Note that since $\Gamma$ is a finite polynomial curve, $\mu=\cH^{1}|_{\Gamma}$ is {\it doubling} on $\Gamma$, meaning there is a constant $C$ so that $\mu(B(x,Mr))\leq C\mu(B(x,r))$ for all $x\in \Gamma$ and $r>0$. If $x\in\d Q$ for some $Q\in \Delta$, then there is a sequence $x_{n}\in X_{n}$ such that $|x_{n}-x|<M^{-n}$ since the $X_{n}$ are maximal $M^{-n}$-nets. To each $x_{n}$ corresponds a ball $B_{n}=B(x_{n},M^{-n})\in \cB_{n}$. Let $N$ be such that $Q\in \Delta_{N}$. Since $cB_{n}\subseteq Q_{B_{n}}\in\Delta_{n}$, we have by \Lemma{cubes} that either $cB_{n}\subseteq Q$ (if $Q_{B_{n}}\cap Q\neq\emptyset$) or $cB_{n}\subseteq R$ for some $R\in \Delta_{N}$ with $Q\cap R=\emptyset$. In either case, since cubes don't contain their boundaries (since they are open), we have that $cB_{n}\cap \d Q=\emptyset$. This implies that $Q$ is porous, and it is well known that such sets have doubling measure zero. More precisely, the doubling condition on $\mu$ guarantees that $\lim_{n\rightarrow\infty} \frac{\mu(\d Q\cap B(x,M^{-n}))}{\mu(B(x,M^{-n}))}=1$ $\mu$-a.e. $x\in \Gamma$ (see \cite[Theorem 1.8]{Heinonen}), but if $x\in \d Q$ and $B_{n}$ is as above, then one can show using the doubling property of $\mu$ that 
\[\limsup_{n\rightarrow\infty} \frac{\mu(\d Q\cap B(x,M^{-n}))}{\mu(B(x,M^{-n}))}
\leq \limsup_{n\rightarrow\infty} \frac{\mu(B(x,M^{-n}) \backslash B_{n})}{\mu(B(x,M^{-n}))}<1,\]
and thus $\mu(\d Q)=0$.

The last part of the theorem follows  easily since $\gamma$ is piecewise affine.

\end{proof}

The following lemma will be used frequently.

\begin{lemma}
Let $I\subseteq \bR$ be an interval, $s:I\rightarrow \ell^{\infty}$ be continuous and $I'\subseteq I$ a subinterval. Then
\begin{equation}
\ell(s|_{I'})-|s(a_{I'})-s(b_{I'})|\leq \ell(s|_{I})-|s(a_{I})-s(b_{I})|.
\label{e:subarc}
\end{equation}
\label{l:subarc}
\end{lemma}

\begin{proof}
We may assume $\ell(s_{I})<\infty$, otherwise \eqn{subarc} is trivial. We estimate
\begin{multline*}
\ell(s|_{I'})-|s(a_{I'})-s(b_{I'})|
= \ell(s|_{I})-\ell(s|_{I\backslash I'}) -|s(a_{I'})-s(b_{I'})|\\
\leq \ell(s|_{I})- (|s(a_{I})-s(a_{I'})|+|s(b_{I})-s(b_{I'})|)-|s(a_{I'})-s(b_{I'})|\\
\leq  \ell(s|_{I})-|s(a_{I})-s(b_{I})|.
\end{multline*}
\end{proof}

\section{Proof of \Theorem{main}}
\label{s:proof}
\subsection{Setup}

For this section, we fix a compact connected set $X$ satisfying the conditions of \Theorem{main}. The main tool is the following Lemma, which can be seen as a very weak substitute for \Theorem{TST}.

\begin{lemma}
Let $c'<\frac{1}{8}$. We can pick $M$ large enough (by picking $\ve>0$ small enough) and pick $\beta_{0},\kappa>0$  such that, for any $X$ satisfying the conditions of \Theorem{main} for some $\beta\in (0,\beta_{0})$,  the following holds. If $X_{n}$ is any nested sequence of $M^{-n}$-nets in $X$, there is $n_{0}=n_{0}(\beta)$ such that for $x_{0}\in X_{n}$ with $M^{-n}<\min\ck{r_{0},\frac{\diam X}{2}}$,
\begin{equation}
 \# X_{n+n_{0}}\cap B(x_{0},c'M^{-n})\geq M^{(1+\kappa\beta^{2})n_{0}}.
\label{e:main-ineq}
\end{equation}
\label{l:lemma-main}
\end{lemma}

We will prove this in Section \ref{s:lemma-main}, but first, we'll explain why this proves \Theorem{main}. 

\begin{proof}[Proof of \Theorem{main}]

Without loss of generality, we may assume $r_{0}>2$ by scaling $X$ if necessary. We first consider the case that $\beta<\beta_{0}$.  Let $\Delta$ be the cubes from \Lemma{cubes} tailored to the metric space $X$ with $c=c'$ and define inductively,
\[\Delta_{0}'=\Delta_{0}, \;\;\; \Delta_{n+1}'=\{R\in \Delta_{(n+1)n_{0}}: R\subseteq Q\mbox{ for some }Q\in \Delta_{n}\}.\]
By \Lemma{lemma-main}, for any $Q\in \Delta_{n}'$, if $B_{Q}=B(x_{Q},cM^{-N})$, then
\begin{equation}
\# \{R\in \Delta_{n+1}',R\subseteq Q\} \geq \# X_{N+n_{0}}\cap Q \geq \#X_{n_{0}}\cap c'B_{Q} \geq M^{(1+\kappa \beta^{2})n_{0}}
\label{e:enough}
\end{equation}
and moreover, since $c'<\frac{1}{8}$,
\begin{equation}
2B_{Q}\cap 2B_{R}=\emptyset \mbox{ for }Q,R\in \Delta_{n}.
\label{e:doubles}
\end{equation}

Define a probability measure $\mu$  inductively by picking $Q_{0}\in \Delta_{0}'$, setting $\mu(Q_{0})=1$ and for $Q\in \Delta_{n}'$ and $R\in\Delta_{n+1}'$, $R\subseteq Q$
\begin{equation}
 \frac{\mu(R)}{\mu(Q)}
 = \frac{1}{\# \{S\in\Delta_{n+1}':S\subseteq Q\}} \stackrel{\eqn{enough}}{ \leq} M^{-(1+\kappa \beta^{2})n_{0}}.
 \label{e:frost}
\end{equation}

Let $x\in X$, $r\in (0,\frac{r_{0}}{M})$. Pick $n$ so that 
\begin{equation}
M^{-n_{0}(n+1)}\leq r<M^{-n_{0}n}.
\label{e:r<M}
\end{equation}
\Claim There is at most one $y\in X_{(n-1)n_{0}}$ such that 
\begin{equation}
B(y,c'M^{-(n-1)n_{0}})\cap B(x,r)\neq\emptyset\;\; \mbox{ and } \;\; Q=Q_{B(y,c'M^{-(n-1)n_{0}})}\in \Delta_{n-1}'.
\label{e:BcapB}
\end{equation}
Indeed, if there were another such $y'\in X_{(n-1)n_{0}}$ with $B(y',c'M^{-(n-1)n_{0}})\cap B(x,r)\neq\emptyset$, then
\begin{multline*}
M^{-(n-1)n_{0}}  \leq |y'-y|\\
  \leq c'M^{-(n-1)n_{0}}+\dist\ps{B(y,c'M^{-(n-1)n_{0}}), B(y',c'M^{-(n-1)n_{0}})} +c'M^{-(n-1)n_{0}}\\
 \leq 2c'M^{-(n-1)n_{0}}+\diam B(x,r) 
 \leq 2c'M^{-(n-1)n_{0}}+2r \\
  \stackrel{\eqn{r<M}}{\leq} 2M^{-(n-1)n_{0}}(c'+M^{-n_{0}})
 <4c'M^{-(n-1)n_{0}} 
 <M^{-(n-1)n_{0}}
\end{multline*}
since $c'<\frac{1}{8}$ and we can pick $\ve<\frac{c'}{8}$ so that $M^{-n_{0}}\leq M^{-1}<c'$, which gives a contradiction and proves the claim. 

Now, assuming we have $y\in X_{(n-1)n_{0}}$ satisfying \eqn{BcapB},
\begin{align*}
B(x,r)
& \subseteq B(y,c'M^{-(n-1)n_{0}}+2r)  
 \stackrel{\eqn{r<M}}{\subseteq} B(y,c'M^{-(n-1)n_{0}}+2M^{-nn_{0}})\\
& \subseteq  B(y,2c'M^{-(n-1)n_{0}})
 =2B_{Q}
\end{align*}
for $M$ large enough (that is, for $2M^{-1}<c'$, which is possible by picking $\ve<\frac{c'}{16}$). If $Q\not\in \Delta_{n-1}'$, then  \eqn{doubles} implies $2B_{Q}\cap 2B_{R}=\emptyset$ for all $R\in \Delta_{n-1}'$, and so
\[\mu(B(x,r))
\leq \mu(2B_{Q})=0.\]
Otherwise, if $Q\in\Delta_{n-1}'$, then $Q\subseteq Q_{0}$, so that
\begin{align*}
\mu(B(x,r))
&  \leq \mu(2B_{Q})
\stackrel{\eqn{doubles}}{=}\mu(Q)
\stackrel{\eqn{frost}}{=} M^{-(1+\kappa \beta^{2})n_{0}(n-1)}\mu(Q_{0})  \stackrel{\eqn{r<M}}{\leq} M^{2(1+\kappa\beta^{2})} r^{-(1+\kappa\beta^{2})}
\end{align*}
thus $\mu$ is a $(1+\kappa\beta^{2})$-Frostmann measure supported on $X$, which implies $\dim X\geq 1+\kappa\beta^{2}$ (c.f. \cite[Theorem 8.8]{Mattila}).

Now we consider the case when $\beta\geq\beta_{0}$. Trivially, $\beta'(x,r)\geq \beta\geq \beta_{0}$ for all $x\in X$ and $r<r_{0}$, and our previous work gives $\dim X\geq 1+\kappa t^{2}$ for all $t<\beta_{0}$, hence $\dim X\geq 1+\kappa \beta_{0}^{2}$. Since $\beta'\leq \frac{1}{2}$, we must have $\beta,\beta_{0}\leq \frac{1}{2}$, and so 
\[\dim X\geq 1+\kappa\beta_{0}^{2}\geq 1+4\kappa\beta_{0}^{2}\beta^{2}\]
and the theorem follows with $c_{0}=4\kappa\beta_{0}^{2}$.

\end{proof}

To show \Lemma{lemma-main}, we will approximate $X$ by a tree containing a sufficiently dense net in $X$ and estimate its length from below. The following lemma relates the length of this tree to the number of net points in $X$.

\begin{lemma}
Let $X_{n_{0}}$ be a maximal $M^{-n_{0}}$-net for a connected metric space $X$ where $n_{0}$ is so that $4M^{-n_{0}}<\frac{\diam X}{4}$. Then we may embed $X$ into $\ell^{\infty}$ so that there is a connected union of finitely many line segments $\Gamma_{n_{0}}\subseteq \ell^{\infty}$  containing $X_{n_{0}}$ such that for any $x\in X_{n_{0}}$ and $r\in (4M^{-n_{0}}, \frac{\diam X}{4})$, 
\begin{equation}
\cH^{1}\ps{\Gamma_{n_{0}}\cap B\ps{x,\frac{r}{2}}}\leq 8M^{-n_{0}} \# (X_{n_{0}}\cap B(x,r)).
\label{e:length-points}
\end{equation}
\label{l:tree}
\end{lemma}

\begin{proof}
Embed $X$ isometrically into $\ell^{\infty}(\bN)$ so that for any $x\in X$, the first $\#X_{n_{0}}$ coordinates are all zero. Construct a sequence of trees $T_{j}$ as follows. Enumerate the elements of $X_{n_{0}}=\{x_{1},...,x_{\# X_{n_{0}}}\}$. For two points $x$ and $y$, let 
\[A_{xy,i}=\{tx+(1-t)y+\max\{t,1-t\}|x-y|e_{i}:t\in [0,1]\}\]
where $e_{i}$ is the standard basis vector in $\ell^{\infty}(\bN)$ (i.e. it is equal to $1$ in the $i$th coordinate and zero in every other coordinate).

Now construct a sequence of trees $T_{j}$ in $\ell^{\infty}(\mathbb{N})$  inductively by setting $T_{0}=\{x_{0}\}$ and $T_{j+1}$ equal to $T_{j}$ united with $S_{j+1}:=A_{x_{j+1}x_{j+1}',j+1}$, where $x_{j+1}'\in \{x_{1},...,x_{j}\}$ and $x_{j+1}\in X_{n_{0}}\backslash \{x_{1},...,x_{j}\}$ are such that 
\[|x_{j+1}-x_{j+1}'|=\dist ( X_{n_{0}}\backslash \{x_{1},...,x_{j}\},\{x_{1},...,x_{j}\}).\]
Since $ X$ is connected, $|x_{j+1}-x_{j+1}'|\leq 2M^{-n_{0}}$, so that
\[\cH^{1}(S_{j})=\cH^{1}(A_{x_{j},x_{j}',j})\leq 2|x_{j}-x_{j}'|\leq 4\cdot 2M^{-n_{0}}=8M^{-n_{0}}.\]
Then $\Gamma_{n_{0}}:=T_{\# X_{n_{0}}}$ is a tree contained in $\ell^{\infty}(\bN)$ containing $X_{n_{0}}$ (the reason we made the arcs $S_{j}$ reach into an alternate dimension is to guarantee that the branches of the tree don't intersect except at the points $X_{n_{0}}$). 

To prove \eqn{length-points}, note that since $\frac{r}{2}>2M^{-n_{0}}$ and 
\[x_{j}\in S_{j}\subseteq B(x_{j},2M^{-n_{0}}),\] 
we have
\begin{align*}
\cH^{1}\ps{\Gamma_{n_{0}}\cap B\ps{x,\frac{r}{2}}}
 \leq \sum_{S_{j}\cap B(x,\frac{r}{2})\neq\emptyset} \cH^{1}(S_{j}) 
& \leq \sum_{x_{j}\in B(x,\frac{r}{2}+2M^{-n_{0}})} 8M^{-n_{0}}\\
& \leq 8\# (X_{n_{0}}\cap B(x,r)).
\end{align*}

\end{proof}

\subsection{Proof of \Lemma{lemma-main}}
\label{s:lemma-main}

We now dedicate ourselves to the proof of \Lemma{lemma-main}. Again, let $ X$ be a connected metric space satisfying the conditions of \Theorem{main}. Without loss of generality, $n=0$, so that $\diam X>2$. Embed $ X$ into $\ell^{\infty}$ as in \Lemma{tree}. Fix $n_{0}\in \bN$. Let $\Gamma_{n_{0}}$ be the tree from \Lemma{tree} containing the $M^{-n_{0}}$-net $X_{n_{0}}\subseteq X$. 

Since $\Gamma_{n_{0}}$ is a tree of finite length that is a union of finitely many line segments, it is not hard to show  that there is a piecewise linear arc length parametrized path $\gamma:[0,2\cH^{1}(\Gamma_{n_{0}})]\rightarrow \Gamma_{n_{0}}$ that traverses almost every point in $\Gamma_{n_{0}}$ at most twice (except at the discrete set of points $X_{n_{0}}$). The proof is similar to that of its graph theoretic analogue.

 Let $\Delta$ be the cubes from \Lemma{cubes} tailored to $\Gamma_{n_{0}}$ and fix $Q_{0}\in \Delta_{0}$. We will adjust the values of $c>0$ in \Lemma{cubes} and the value $\ve>0$ in the definition of $M$ as we go along the proof. Note that $\diam X>2$ implies $\diam \Gamma_{n_{0}}>1>(1+\ve\beta)c$ if $c<\frac{1}{8}$, and so $\Gamma_{n_{0}}\not\subseteq Q_{0}$.  
\def\twl{\tilde{\cL}}
For $Q,R\in \Delta$, write $R^{1}=Q$ if $R$ is a maximal cube in $\Delta$ properly contained in $Q$. For $n\geq 0$ and $Q\in \Delta$, define
\[\cL_{1}(Q)=\{R\in \Delta: R^{1}=Q\}, \;\;\; \cL_{n}(Q)=\bigcup_{R\in \cL_{n-1}(Q)}\cL_{1}(R),\]
\[ \twl_{n}(Q)=\cL_{n}(Q)\cap \bigcup_{j=0}^{n_{0}-1}\Delta_{j}, \;\;\; \twl(Q)=\bigcup \twl_{n}(Q)\]
\[\twl_{n}=\twl_{n}(Q_{0}), \;\;\twl=\twl(Q_{0}).\]

For $Q\in \Delta$, let 
\[\lambda(Q)=\{[a,b]: (a,b)\mbox{ is a connected component of }\gamma^{-1}(Q)\}\]
and for $n\leq n_{0}$, define $\gamma_{n}$ to be the continuous function such that for all $Q\in \cL_{n}(Q_{0})$ and $[a,b]\in \lambda(Q)$,
\[\gamma_{n}|_{[a,b]}(at+(1-t)b)=t\gamma(a)+(1-t)\gamma(b)\mbox{ for }t\in [0,1],\]
that is, $\gamma_{n}$ is linear in all cubes in $\Delta_{n}$ and agrees with $\gamma$ on the boundaries of the cubes (see Figure \ref{f:cubes}).

%\begin{figure}[h]
% \begin{picture}(100,200)(-70,0)
%%\putgrid
%\put(0,0){\scalebox{.23}{\includegraphics{cubes.pdf}}}
%\put(65,105){(a)}
%\put(0,0){(b)}
%\put(125,0){(c)}
%\put(45,170){$Q$}
%\put(170,175){$R\in \cL_{1}(Q)$} 
%\put(195,105){$\gamma_{n+1}|_{I}$}
%\put(20,110){$\gamma|_{I}$}
%\put(105,15){$\gamma_{n}|_{I}$}
%\end{picture}
%\caption{In (a), we have a typical cube $Q\in \Delta_{n}$, and some of its children in $\cL_{1}(Q)$. Observe that their sizes can be radically different. In (b) are the components $\gamma|_{\gamma^{-1}(Q)}$, where in this case $\gamma^{-1}(Q)$ consists of two intervals, and we've pointed at a particular component $\gamma|_{I}$ for some $I\in \lambda(Q)$.  In (c), the dotted lines represent the components of $\gamma_{n}|_{\gamma^{-1}(Q)}$, which is affine in cubes in $\Delta_{n}$, and hence is affine in $Q$, and the components of $\gamma_{n+1}|_{\gamma^{-1}(Q)}$, which are affine in the children of $Q$ (since they are in $\Delta_{n+1}$). }
%\label{f:cubes}
%\end{figure}

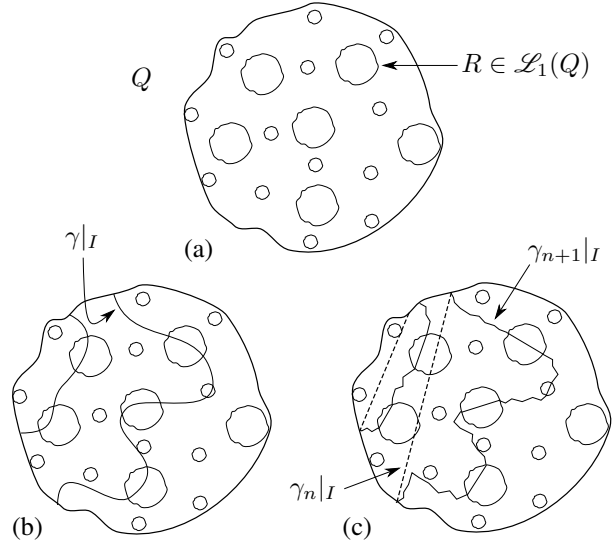
\begin{figure}[h]
 \begin{picture}(100,200)(60,0)
%\putgrid
\put(0,0){\scalebox{.23}{\includegraphics{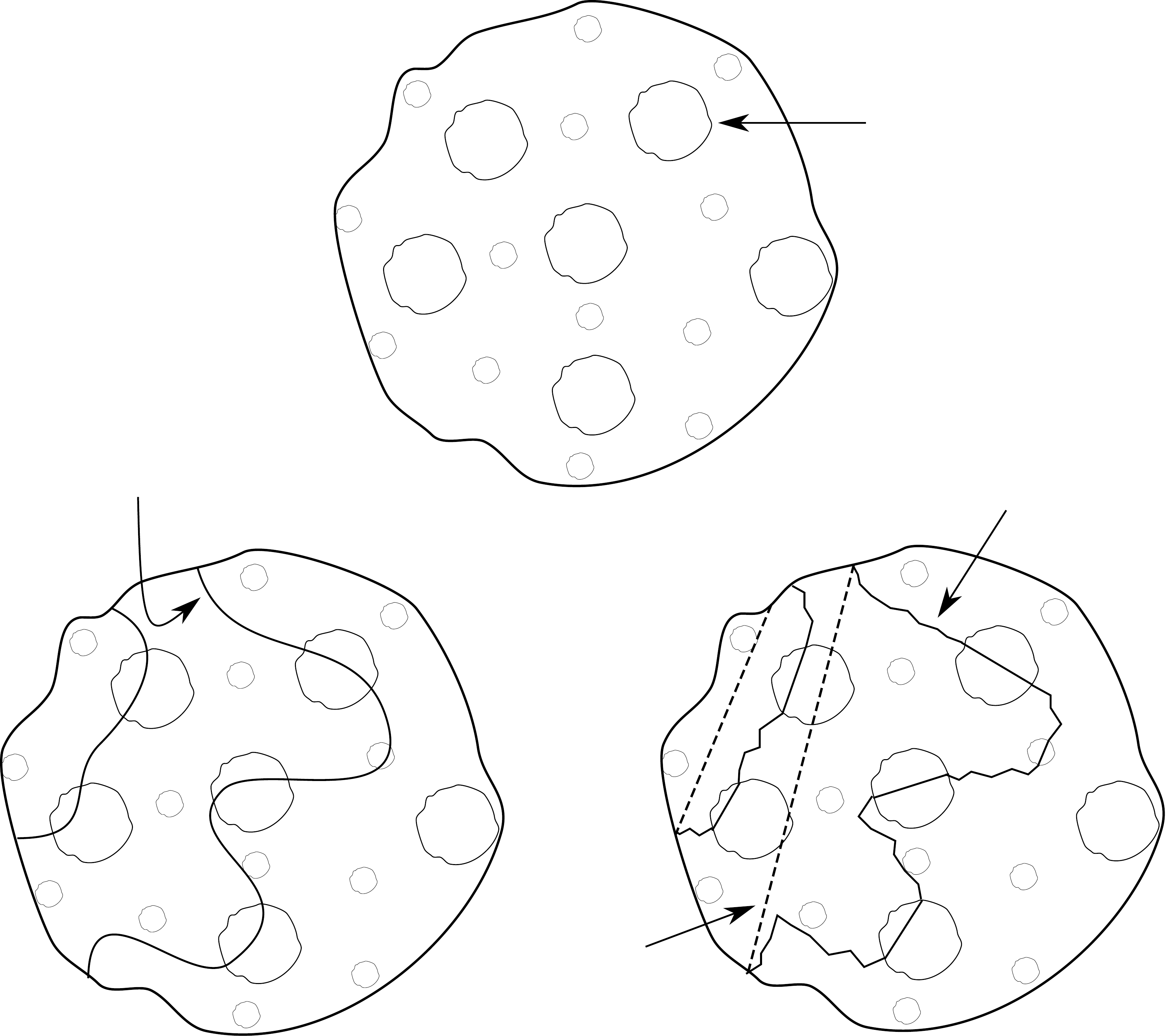}}}
\put(65,105){(a)}
\put(0,0){(b)}
\put(125,0){(c)}
\put(45,170){$Q$}
\put(170,175){$R\in \cL_{1}(Q)$} 
\put(195,105){$\gamma_{n+1}|_{I}$}
\put(20,110){$\gamma|_{I}$}
\put(105,15){$\gamma_{n}|_{I}$}
\end{picture}
\caption{In (a), we have a typical cube $Q\in \Delta_{n}$, and some of its children in $\cL_{1}(Q)$. Note that their sizes can be radically different. In (b) are the components $\gamma|_{\gamma^{-1}(Q)}$, where in this case $\gamma^{-1}(Q)$ consists of two intervals, and we've pointed at a particular component $\gamma|_{I}$ for some $I\in \lambda(Q)$.  In (c), the dotted lines represent the components of $\gamma_{n}|_{\gamma^{-1}(Q)}$, which is affine in cubes in $\Delta_{n}$, and hence is affine in $Q$, and the solid piecewise-affine curves represent the components of $\gamma_{n+1}|_{\gamma^{-1}(Q)}$, which are affine in the children of $Q$ (since they are in $\Delta_{n+1}$). }
\label{f:cubes}
\end{figure}

\Lemma{lemma-main} will follow from the following two lemmas:

\begin{lemma}
There is $K\in (0,1)$ and $\beta_{0}>0$ (independent of $n_{0}$ above) such that if $\beta\in (0,\beta_{0})$, $n<n_{0}$, and $Q\in \twl_{n}$, either
\begin{equation}
\sum_{I\in \lambda(Q)}(\ell(\gamma_{n+1}|_{I})- \ell(\gamma_{n}|_{I})) \geq \frac{\ve\beta}{4}\diam Q
\label{e:eb/4}
\end{equation}
or $Q\in \Delta_{Bad}$, where
\begin{equation}
 \Delta_{Bad}=\{R\in\twl: \cH^{1}_{\infty}(\Gamma_{n_{0}}\cap R) \geq (1+K\beta)\diam R\}
 \label{e:bad}
 \end{equation}
\label{l:good-or-bad}
\end{lemma}

\begin{lemma}
With $\Delta_{Bad}$ defined as above, we have
\begin{equation}
\sum_{Q\in \Delta_{Bad}}\beta\diam Q \leq \frac{2}{K} \cH^{1}(\Gamma_{n_{0}}).
\label{e:bad}
\end{equation}
\label{l:bad}
\end{lemma}

We'll prove these in sections \ref{s:good-or-bad} and \ref{s:lemma-bad} respectively, but first let us finish the proof of \Lemma{lemma-main}. \\

For $Q\in \twl$, let $n(Q)$ be such that $Q\in \cL_{n}$ and define
\[d(Q)= \sum_{I\in \lambda(Q)} \ps{\ell(\gamma_{n(Q)+1}|_{I})-\ell(\gamma_{n(Q)}|_{I})}.\]

By telescoping sums and \Lemma{zero-boundary}, we have
\begin{align}
\sum_{Q\in \twl}d(Q)
& =\sum_{n=0}^{n_{0}-1}\sum_{Q\in\twl_{n}}\sum_{I\in \lambda(Q)}\ps{\ell(\gamma_{n+1}|_{I})-\ell(\gamma_{n}|_{I}))} \notag \\
& =\sum_{n=0}^{n_{0}-1} \ps{\ell(\gamma_{n+1}|_{\gamma^{-1}(Q_{0}))}-\ell(\gamma_{n}|_{\gamma^{-1}(Q_{0}))}} \notag \\
& \leq \ell(\gamma|_{\gamma^{-1}(Q_{0})})= 2\cH^{1}(\Gamma_{n_{0}}\cap Q_{0}).
\label{e:sumdQ}
\end{align}
Note that $\diam (\Gamma_{n_{0}}\cap Q_{0})\geq 1$ since $Q_{0}\in \Delta_{0}$, $\diam \Gamma_{n_{0}}>1$, and $\Gamma_{n_{0}}$ is connected. This,  \Lemma{good-or-bad}, and \Lemma{bad} imply
\begin{align*}
\frac{10}{K\ve}\cH^{1} & (\Gamma_{n_{0}}\cap Q_{0})
 \geq \frac{2}{K\ve} \cH^{1}(\Gamma_{n_{0}}\cap Q_{0}) + \frac{8}{\ve}\cH^{1}(\Gamma_{n_{0}}\cap Q_{0}) \notag \\
& \stackrel{\eqn{bad} \atop \eqn{sumdQ}}{\geq} \sum_{Q\in \Delta_{Bad}}\beta \diam Q + \frac{4}{\ve}\sum_{Q\in \twl\backslash\Delta_{Bad}} d(Q)  \notag \\
  & \stackrel{\eqn{eb/4}}{\geq} \sum_{Q\in \Delta_{Bad}}\beta \diam Q+ \sum_{Q\in \twl\backslash \Delta_{Bad}}\beta \diam Q =\sum_{Q\in\tilde{\cL}}\beta\diam Q \notag \\
& = \sum_{n=0}^{n_{0}-1}\sum_{Q\in \Delta_{n}}\beta\diam Q
\geq  \sum_{n=0}^{n_{0}-1}\sum_{Q\in \Delta_{n}}\beta\diam B_{Q} \notag \\
& =   \sum_{n=0}^{n_{0}-1}c\sum_{Q\in \Delta_{n}}\beta\diam \frac{1}{c}B_{Q}  
\stackrel{\eqn{1/cQ}}{\geq}cn_{0} \beta \diam (\Gamma_{n_{0}}\cap Q_{0})
 \geq cn_{0}\beta
\end{align*}
so that
\[\frac{Kcn_{0}\beta \ve}{10}\leq \cH^{1}(\Gamma_{n_{0}}\cap Q_{0}).\]
By \Lemma{tree}, and since $B_{Q_{0}}$ has radius $c$,
\begin{align*}
\cH^{1}(\Gamma_{n_{0}}\cap Q_{0})
& \leq \cH^{1}(\Gamma_{n_{0}} 
 \cap (1+\ve\beta)B_{Q_{0}})
\leq \cH^{1}(\Gamma_{n_{0}}\cap B(x,2c)) \\
& \leq 8 \#(X_{n_{0}}\cap B(x,4c))M^{-n_{0}}
\end{align*}
Combining these two estimates we have, for $c<\frac{c'}{4}$ that

\[ \delta n_{0} M^{n_{0}}  \beta \leq \#(X_{n_{0}}\cap B(x_{0},c')), \;\;\; \delta=\frac{Kc\ve}{80} \]
Pick $n_{0}=\ceil{\frac{8}{\delta\beta^{2}\ve}}$. Since $\frac{8}{\ve\beta}=M$, we get
 \begin{multline*}
 \#(X_{n_{0}}\cap B(x_{0},c'))
 \geq \delta n_{0}M^{n_{0}}\beta
=n_{0}\ps{\frac{\delta \ve \beta^{2}}{8}}M^{n_{0}}\frac{8}{\ve\beta}
\geq M^{n_{0}+1}\\
=M^{n_{0}(1+\frac{1}{n_{0}})}
\geq M^{n_{0}(1+\frac{1}{\frac{8}{\delta\beta^{2}}-1})}
\geq M^{n_{0}(1+\frac{\delta}{16}\beta^{2})}
\end{multline*}
since $\frac{8}{\delta\beta^{2}}\geq 2$, and this proves \Lemma{lemma-main} with $\kappa=\frac{\delta}{16}$.\\

\begin{remark}
By inspecting the proof of \Lemma{good-or-bad} below, one can solve for explicit values of $\ve,c,\beta_{0}$, and $K$. In particular, one can choose $\ve<\frac{1}{12288}$, $K<\frac{1}{4096}$, $c<\frac{1}{64}$, and $\beta_{0}=\frac{1}{356}$, so that the supremum of permissible values of $\kappa$ is at least $ 2^{-41}$, and is by no means tight.
\end{remark}

In the next two subsections, we prove \Lemma{good-or-bad} and \Lemma{bad}.

\subsection{Proof of \Lemma{good-or-bad}}
\label{s:good-or-bad}

Fix $Q$ as in the statement of the lemma. For any $I\in \lambda(Q)$, 
\begin{align*}
\ell(\gamma_{n+1}|_{I})-\ell(\gamma_{n}|_{I})
& \geq \ell(\gamma_{n+1}|_{I})- |\gamma_{n}(a_{I})-\gamma_{n}(b_{I})|\\
& =  \ell(\gamma_{n+1}|_{I})- |\gamma_{n+1}(a_{I})-\gamma_{n+1}(b_{I})|\geq 0.
\end{align*}
Hence, to prove the lemma,  it suffices to show that either $Q\in\Delta_{Bad}$ or there is an interval $I\in \lambda(Q)$ for which
\[\ell(\gamma_{n+1}|_{I})-\ell(\gamma_{n}|_{I})\geq \frac{\ve\beta}{4}\diam Q.\]
Fix $N$ so that $Q\in \Delta_{N}$. Let $\tilde{Q}\in \Delta_{N+1}$ be such that 
\[x_{Q}\in \tilde{Q}\subset\tilde{Q}^{1}=Q\]  
and pick $I\in\lambda(Q)$ such that $\gamma_{n+1}(I)\cap \tilde{Q}\neq\emptyset$. Note that $\gamma_{n}|_{I}\subseteq Q$ is a segment with endpoints the same as $\gamma_{n+1}|_{I}$,  hence
\begin{align}
\ell(\gamma_{n}|_{I})
& =\cH^{1}(\gamma_{n}(I))
=\diam \gamma_{n}(I) 
 =|\gamma_{n}(a_{I})-\gamma_{n}(b_{I})|\notag \\
& =|\gamma_{n+1}(a_{I})-\gamma_{n+1}(b_{I})| \leq \diam Q
\label{e:allthesame}
\end{align} 

Before proceeding, we'll give a rough idea of how the proof will go.  We will consider a few cases, which are illustrated in Figure \ref{f:cases} below.\\

\begin{figure}[h]

\begin{picture}(100,240)(80,0)
\put(0,0){{\includegraphics[width=240pt]{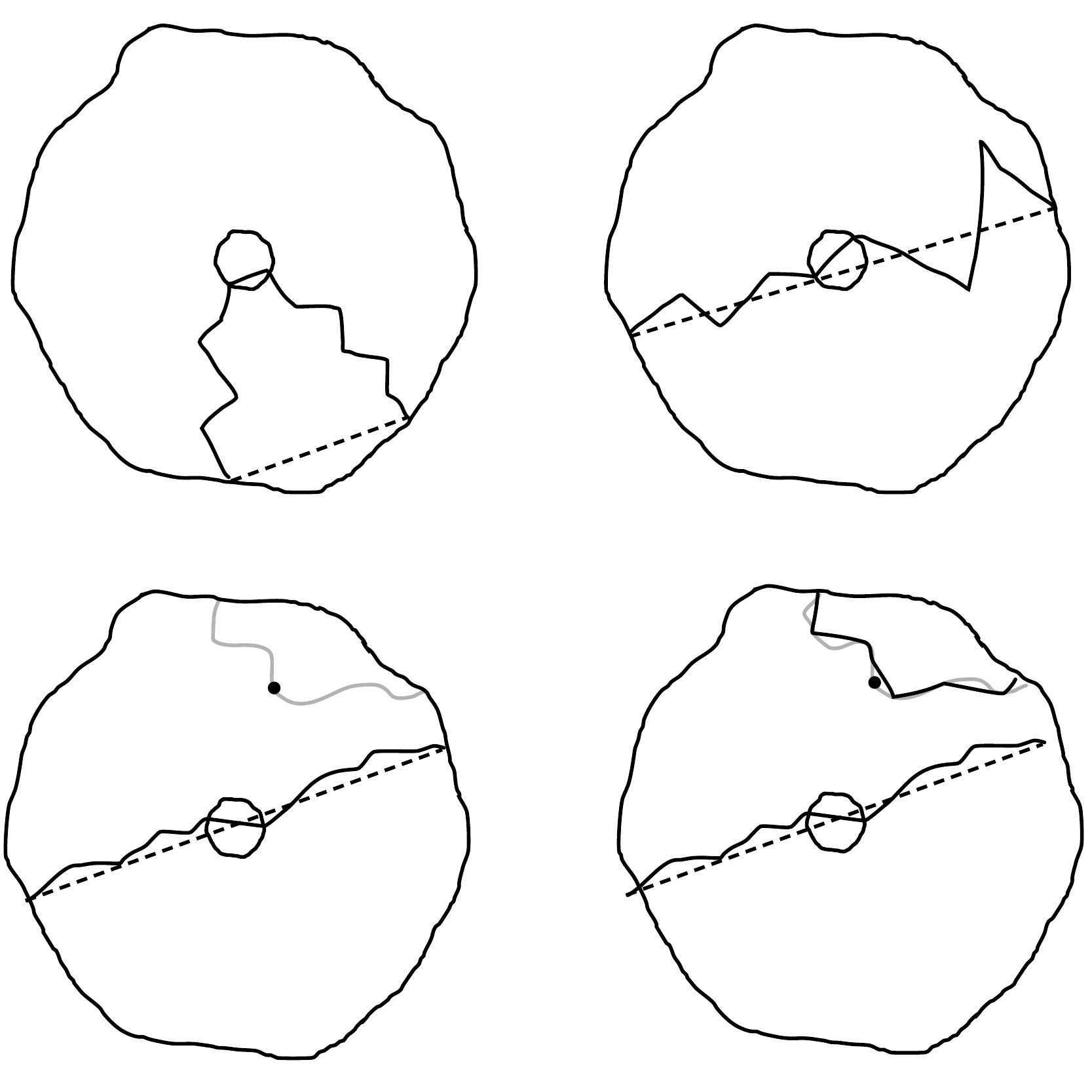}}}
%\putgrid
\put(55,145){$\gamma_{n}(I)$}
\put(70,175){$\gamma_{n+1}(I)$}
\put(40,190){$\tilde{Q}$}
\put(10,220){$Q$}
\put(52,85){$z$}
\put(35,95){{}\color{gray} $ X$}
\put(200,80){$\rho$}
\put(0,125){Case 1}
\put(125,125){Case 2a}
\put(0,0){Case 2b}
\put(185,85){$z$}
\put(125,0){Case 2b cont.}
\end{picture}
\caption{Illustrations of cases 1,2a, and 2b.}
\label{f:cases}
\end{figure}
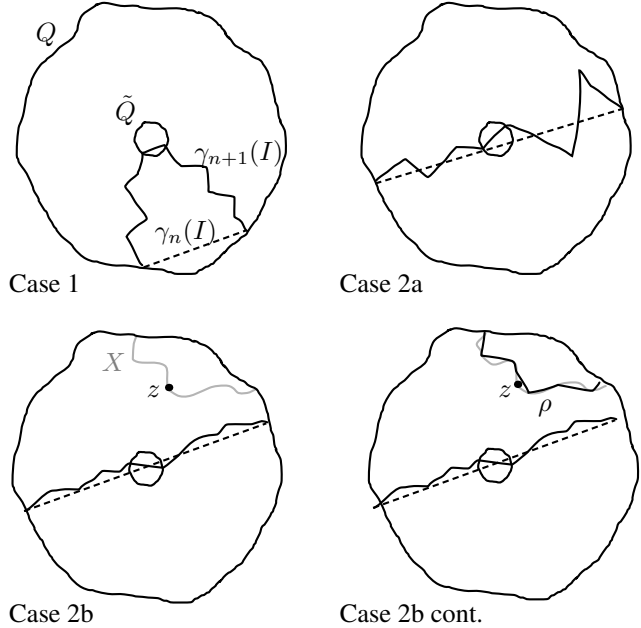

In the first case, we assume the diameter of $\gamma_{n}(I)$ is small with respect to $Q$; since $\gamma_{n+1}|_{I}$ has the same endpoints as $\gamma_{n}|_{I}$ and intersects the center cube $\tilde{Q}$, there must be a large difference in length between $\gamma_{n+1}(I)$ and $\gamma_{n}(I)$ since the former must enter $Q$, hit $\tilde{Q}$, and then exit $Q$, and so \eqn{eb/4} will hold. For the next two cases, we assume $\gamma_{n}(I)$ has large diameter. The second case (2a) assumes that $\gamma_{n+1}(I)$ contributes more length than $\gamma_{n}(I)$, again implying \eqn{eb/4} trivially. (It is possible to combine this case with (1), but we found this split to be somewhat convenient.) In the final case (2b) we assume the difference in length between $\gamma_{n+1}(I)$ and $\gamma_{n}(I)$ is small. Since $\beta_{X}(B_{Q})>\beta$, we can show this implies the existence of $z\in  X$ far away from $\gamma_{n+1}(I)$ (since $\gamma_{n+1}|_{I}$ has small geodesic deviation, so it can't approximate all of $ X$ in $B_{Q}$). Since $\Gamma_{n_{0}}$ approximates $ X$, we can find a large curve $\rho\subseteq \Gamma_{n_{0}}$ entering $B_{Q}$, approaching $z$, and then leaving $B_{Q}$. The presence of both $\gamma(I)$ and $\rho$ inside $Q$ implies that the total length of $\Gamma_{n_{0}}\cap Q$ must be large, which means $Q\in \Delta_{Bad}$. \\

Now we proceed with the actual proof.

{\bf Case 1:} Suppose $\ell(\gamma_{n}(I))<\frac{\diam Q}{4}$. Since $\gamma_{n+1}|_{I}$ is a path entering $Q$, hitting $\tilde{Q}$, and then leaving $Q$, we can estimate

\begin{align}
\ell(\gamma_{n+1}|_{I})
& \geq 2\dist(\tilde{Q},Q^{c})
\stackrel{\eqn{1+veb}}{\geq} 2\dist ((1+\ve\beta)B_{\tilde{Q}},B_{Q}) \notag \\
& =2(cM^{-N}-(1+\ve\beta)cM^{-N-1})
 =2cM^{-N}(1-(1+\ve\beta)M^{-1})\notag \\
& \geq \diam B_{Q} \ps{1-\frac{\ve\beta}{8}-\frac{\ve^{2}\beta^{2}}{8}}
 >(1-\ve\beta)\diam B_{Q} \notag \\ 
& \stackrel{\eqn{1+veb}}{\geq} \frac{1-\ve\beta}{1+\ve\beta}\diam Q  =\ps{\frac{1+\ve\beta}{1+\ve\beta}-\frac{2\ve\beta}{1+\ve\beta}}\diam Q  \geq (1-2\ve\beta)\diam Q.
\label{e:ln+1}
\end{align}
Thus,
\[
\ell(\gamma_{n+1}|_{I})-\ell(\gamma_{n}|_{I})
\stackrel{\eqn{ln+1}}{\geq} (1-2\ve\beta)\diam Q-\frac{\diam Q}{4} 
\geq \frac{\diam Q}{8}
\]
if $\ve<\frac{1}{16}$, which implies the lemma in this case.\\

{\bf Case 2:} Suppose
\begin{equation}
\ell(\gamma_{n}|_{I})\geq \frac{\diam Q}{4} 
\label{e:case2}
\end{equation} 
We again split into two cases.\\

{\bf Case 2a:} Suppose
\[ \ell(\gamma_{n+1}|_{I})\geq (1+\ve \beta) \ell(\gamma_{n}|_{I}).\]
Then
\[\ell(\gamma_{n+1}|_{I}) - \ell(\gamma_{n}|_{I}) \geq \ve \beta \ell(\gamma_{n}|_{I}) \stackrel{\eqn{case2}}{\geq} \frac{\ve\beta}{4}\diam Q.\]

{\bf Case 2b:} Now suppose 
\begin{equation}
\ell(\gamma_{n+1}|_{I})< (1+\ve \beta) \ell(\gamma_{n}(I)).
\label{e:2b}
\end{equation}

Note that in this case, we have a better lower bound on $\ell(\gamma_{n}|_{I})$, namely,
\begin{equation}
\ell(\gamma_{n}|_{I})
\stackrel{\eqn{2b}}{\geq} \frac{\ell(\gamma_{n+1}|_{I})}{1+\ve\beta}
\stackrel{\eqn{ln+1}}{\geq} \frac{1-2\ve\beta}{1+\ve\beta}\diam Q
\geq (1-3\ve\beta)\diam Q.
\label{e:ln}
\end{equation}

Let $C\in (0,1)$ (we will pick its value later). 

\begin{sublemma}
Assuming the conditions in case 2b, let $I'\subseteq I$ be the smallest interval with 
\[\gamma_{n+1}(a_{I'}),\gamma_{n+1}(b_{I'})\in \d ((1-C\beta) B_{Q})\] 
and $\gamma_{n+1}(I')\cap \tilde{Q}\neq\emptyset$. Then
\begin{equation}
\ell(\gamma_{n+1}|_{I'})-|\gamma_{n+1}(a_{I'})-\gamma_{n+1}(b_{I'})| \leq 2\ve\beta |\gamma_{n+1}(a_{I'})-\gamma_{n+1}(b_{I'})|
\label{e:betaI'}
\end{equation}
\end{sublemma}

\begin{proof}
Since $\gamma_{n+1}$ enters $(1-C\beta)B_{Q}$, hits $\tilde{Q}$, and then leaves $(1+C\beta)B_{Q}$, we have
\begin{align}
\ell(\gamma_{n+1}|_{I'})
&\geq 2\dist(\tilde{Q},(1-C\beta)B_{Q}^{c})
 \stackrel{\eqn{1+veb}}{\geq} 2\dist ((1+\ve\beta)B_{\tilde{Q}},(1-C\beta)B_{Q}^{c}) \notag \\
& = 2( (1-C\beta)cM^{-N}-(1+\ve\beta)cM^{-N-1}) \notag \\
& =2cM^{-N}(1-C\beta-(1+\ve\beta)M^{-1})  > \diam B_{Q}(1-C\beta-2M^{-1}) \notag \\ 
&  = (1-C\beta-\frac{\ve\beta}{4})\diam B_{Q} 
 \stackrel{\eqn{1+veb}}{\geq} \frac{1-C\beta-\frac{\ve\beta}{4}}{1+\ve\beta}\diam Q  \notag \\
 & = \ps{\frac{1+\ve\beta}{1+\ve\beta}-\frac{C\beta+\frac{5\ve\beta}{4}}{1+\ve\beta}}\diam Q 
   > (1-C\beta-2\ve\beta)\diam Q
\label{e:ln+1'}
\end{align}
Hence,
%\begin{align}
%\max & \{ |\gamma_{n+1}(a_{I})-\gamma_{n+1}(a_{I'})|,|\gamma_{n+1}(b_{I})-\gamma_{n+1}(b_{I'})|\} \notag \\ 
%& \leq \ell(\gamma_{n+1}|_{I\backslash I'}) 
% =\ell(\gamma_{n+1}|_{I})-\ell(\gamma_{n+1}|_{I'}) \notag \\ 
% & \stackrel{\eqn{1+veb},\eqn{ln+1'}}{\leq} (1+\ve\beta)\ell(\gamma_{n}(I))-(1-C\beta-2\ve\beta)\diam Q\notag \\
%& \leq (1+\ve\beta)\diam Q-(1-C\beta-2\ve\beta)\diam Q\notag \\
%& = (3\ve\beta+C\beta)\diam Q.
%\label{e:maxgam}
%\end{align}
\begin{align}
|\gamma_{n+1} & (a_{I})  -\gamma_{n+1}(b_{I})| - |\gamma_{n+1}(a_{I'})-\gamma_{n+1}(b_{I'})| \notag  \\
& \leq |\gamma_{n+1}(a_{I})-\gamma_{n+1}(a_{I'})|+|\gamma_{n+1}(b_{I})-\gamma_{n+1}(b_{I'})| \notag \\
 & \leq \ell(\gamma_{n+1}|_{I\backslash I'}) \notag  =\ell(\gamma_{n+1}|_{I})-\ell(\gamma_{n+1}|_{I'})  \notag \\ 
&   \stackrel{\eqn{2b} \atop \eqn{ln+1'}}{\leq} (1+\ve\beta)\ell(\gamma_{n}(I))-(1-C\beta-2\ve\beta)\diam Q \notag \\
 & \stackrel{\eqn{ln+1}}{\leq} (1+\ve\beta)\diam Q-(1-C\beta-2\ve\beta)\diam Q \notag \\
 & = (3\ve\beta+C\beta)\diam Q \label{e:3veb+Cb} 
  \stackrel{\eqn{allthesame}\atop\eqn{case2}}{ \leq} 4(3\ve\beta +C\beta)|\gamma_{n+1}(a_{I})-\gamma_{n+1}(b_{I})|
\end{align}
Thus,
\begin{align}
|\gamma_{n+1}(a_{I})-\gamma_{n+1}(b_{I})|  & \leq \frac{|\gamma_{n+1}(a_{I'})-\gamma_{n+1}(b_{I'})|}{1-4(3\ve\beta +C\beta)}\notag \\
& \leq 2|\gamma_{n+1}(a_{I'})-\gamma_{n+1}(b_{I'})|
\label{e:aba'b'}
\end{align}
if we pick $\ve<\frac{1}{24}$ and $\beta<\frac{1}{8}$ (recall $C\in(0,1)$). By \Lemma{subarc},

\begin{multline*}
\ell(\gamma_{n+1}|_{I'})-|\gamma_{n+1}(a_{I'})-\gamma_{n+1}(b_{I'})| 
\stackrel{\eqn{subarc}}{\leq} \ell(\gamma_{n+1}|_{I})- |\gamma_{n+1}(a_{I})-\gamma_{n+1}(b_{I})|\\
\stackrel{\eqn{2b}}{<} \ve\beta |\gamma_{n+1}(a_{I})-\gamma_{n+1}(b_{I})|
\stackrel{\eqn{aba'b'}}{\leq} 2\ve\beta |\gamma_{n+1}(a_{I'})-\gamma_{n+1}(b_{I'})|
\end{multline*}
which proves \eqn{betaI'}.
\end{proof}

By the main assumption in \Theorem{main}, and because we're assuming $n=0$ so that $M^{-n}=1<r_{0}$,
\begin{multline*}
\beta
<\beta_{X}'(x_{Q},(1-C\beta)cM^{-N})\\
\leq \frac{\ell(\gamma_{n+1}|_{I'})-|\gamma_{n+1}(a_{I'})-\gamma_{n+1}(b_{I'})| + \sup_{z\in (1-C\beta)B_{Q}\cap X}\dist(z,\gamma_{n+1}(I'))}{|\gamma_{n+1}(a_{I'})-\gamma_{n+1}(b_{I'})|}\\
\stackrel{\eqn{betaI'}}{\leq} \frac{2\ve\beta |\gamma_{n+1}(a_{I'})-\gamma_{n+1}(b_{I'})|+\sup_{z\in (1-C\beta)B_{Q}\cap X}\dist(z,\gamma_{n+1}(I'))}{|\gamma_{n+1}(a_{I'})-\gamma_{n+1}(b_{I'})|}\\
= 2\ve\beta+\frac{\sup_{z\in (1-C\beta)B_{Q}\cap X}\dist(z,\gamma_{n+1}(I'))}{|\gamma_{n+1}(a_{I'})-\gamma_{n+1}(b_{I'})|}
\end{multline*}
so there is $z\in X\cap (1-C\beta)B_{Q}$ with 
\begin{align}
\dist(z,\gamma_{n+1}(I')) 
& \geq (\beta-2\ve\beta) |\gamma_{n+1}(a_{I'})-\gamma_{n+1}(b_{I'})| \notag \\
& \stackrel{\eqn{aba'b'}}{\geq} \frac{\beta-2\ve\beta}{2}|\gamma_{n+1}(a_{I})-\gamma_{n+1}(b_{I})| \notag \\
& \stackrel{\eqn{case2}}{\geq} \frac{\beta-2\ve\beta}{8}\diam Q  
 \geq \frac{\beta}{16}\diam Q
 \label{e:b/16}
\end{align}
if $\ve<\frac{1}{4}$.

%\begin{figure}[h]
%\begin{picture}(100,200)(125,0)
%%\putgrid
%\put(0,0){\includegraphics[width=350pt]{gamman.pdf}}
%\put(75,100){$\tilde{Q}$}
%\put(25,150){$Q$}
%\put(280,113){$z$}
%\put(25,85){$\gamma_{n+1}(I)$}
%\put(110,60){$\gamma_{n}(I)$}
%\put(225,85){$\gamma_{n+1}(I')$}
%\put(250,140){$\rho$}
%\put(240,30){$(1-C\beta)B_{Q}$}
%\put(290,140){{\color{gray} $ X$}}
%\put(280,93){$\zeta$}
%
%\end{picture}
%\caption{}
%\label{f:z}
%\end{figure}

%
Since $\gamma_{n+1}([0,1])$ hits every cube in $\cL_{1}(Q)$, which all have diameter at most $2(1+\ve\beta)cM^{-N-1}$ by \eqn{1+veb} (recall $N$ was chosen so that $Q\in \Delta_{N}$), 
\[
\Gamma_{n_{0}}\cap Q\subseteq (\gamma_{n+1}([0,1]))_{2(1+\ve\beta)cM^{-N-1}} \subseteq  (\gamma_{n+1}([0,1]))_{4cM^{-N-1}}  \]
Note that since $Q\in \twl_{n}$, we have $N<n_{0}$. Since $X_{n_{0}}\subseteq \Gamma_{n_{0}}\cap  X$ and $N<n_{0}$,
\begin{align*}
X\cap (1-C\beta)B_{Q}  
& \subseteq X \cap Q
\subseteq (\Gamma_{n_{0}}\cap Q)_{2M^{-n_{0}}}
 \subseteq (\gamma_{n+1}([0,1]))_{4cM^{-N-1}+ 2M^{-n_{0}}}\\
& \subseteq (\gamma_{n+1}([0,1]))_{(4cM^{-N-1} + 2M^{-N-1})}
 =  (\gamma_{n+1}([0,1]))_{(2+\frac{1}{c})M^{-1}2cM^{-N}}\\
& =(\gamma_{n+1}([0,1]))_{(2+\frac{1}{c})M^{-1}\diam B_{Q}}
 \subseteq (\gamma_{n+1}([0,1]))_{\frac{2}{c}M^{-1}\diam B_{Q}}
\end{align*}
since $c<\frac{1}{8}$. Since $z\in X\cap (1-C\beta)B_{Q}$, there is $t\in [0,1]$ such that 
\begin{equation}
 |\gamma_{n+1}(t)-z|<\frac{2}{c}M^{-1} \diam B_{Q} = \frac{\ve\beta}{4c}\diam Q
 \label{e:gt-z}
 \end{equation}
and so
\begin{multline}
\dist(\gamma_{n+1}(t),\gamma_{n+1}(I'))
\geq \dist (z,\gamma_{n+1}(I'))-|\gamma_{n+1}(t)-z|\\
\stackrel{\eqn{b/16} \atop \eqn{gt-z}}{\geq} \ps{\frac{\beta}{16}-\frac{\ve\beta}{4c}}\diam Q\geq \frac{\beta}{32} \diam Q
\label{e:b/32}
\end{multline}
for $\ve<\frac{c}{8}$. Also, since $z\in (1-C\beta)B_{Q}$, we know that

\begin{align}
B_{Q} & 
  \supseteq B\ps{z,\frac{C\beta}{2} \diam B_{Q}}
 \hspace{-3pt} \stackrel{\eqn{1+veb}}{\supseteq} B\ps{z,\frac{C\beta}{2(1+\ve\beta)}\diam Q}    \notag \\
&  \supseteq B\ps{z, \frac{C\beta}{4}  \diam Q} \hspace{-5pt}\stackrel{\eqn{gt-z}}{\supseteq} B\ps{\gamma_{n+1}(t),\ps{\frac{C\beta}{4}-\frac{\ve\beta}{4c}}\diam Q} \notag \\
& \supseteq B\ps{\gamma_{n+1}(t),\frac{C\beta}{8}\diam Q}
\end{align}
for $\ve<\frac{Cc}{2}$. In particular, $t\in \gamma_{n+1}^{-1}(B_{Q})$. Note

\begin{align*}
\dist & (\gamma_{n+1}(t),\gamma_{n+1}(I)) \notag \\
& \geq \dist(\gamma_{n+1}(t),\gamma_{n+1}(I')) 
-\max\{\diam\gamma([a_{I},a_{I}']),\diam \gamma([b_{I}',b_{I}])\} \notag \\
& \geq \dist(\gamma_{n+1}(t),\gamma_{n+1}(I'))-\ell(\gamma|_{I/I'})
\stackrel{\eqn{3veb+Cb} \atop \eqn{b/32}}{\geq} \frac{\beta}{32}\diam Q-(3\ve\beta+C\beta)\diam Q \notag \\
& \geq \frac{\beta}{64}\diam Q\end{align*}
for $\ve<\frac{1}{384}$ and $C<\frac{1}{128}$. Thus, since of course $\frac{C}{8}<\frac{1}{128}$, we have

\[B\ps{\gamma_{n+1}(t),\frac{C\beta}{8}\diam Q}\subseteq Q\backslash (\gamma_{n+1}(I))_{\frac{\beta}{128}\diam Q}\]
In particular, $\gamma_{n+1}(t)\in Q$, and so by construction, $t\in [a,b]$ for some $[a,b]\in \lambda(Q)$, where $\gamma_{n+1}(a)$ and $\gamma_{n+1}(b)$ are both in $\Gamma_{n_{0}}$. In particular, $\gamma_{n+1}((a,b))$ is a line segment in a cube $R\in \twl_{1}(Q)$. If $\zeta:=\gamma_{n+1}(a)\in \Gamma_{n_{0}}$, then
\begin{align}
|\zeta & -\gamma_{n+1}(t)|
 \leq \diam R \stackrel{\eqn{1+veb}}{\leq} (1+\ve\beta)\diam B_{R}
 =2(1+\ve\beta)cM^{-N-1} \notag \\
& \leq (1+\ve\beta)M^{-1}\diam Q 
 = (1+\ve\beta)\frac{\ve\beta}{8}\diam Q 
  \leq \frac{\ve\beta}{4}\diam Q 
   \leq \frac{C\beta}{16}\diam Q 
\label{e:zeta-gamma}
\end{align}
%Hence,
% \begin{align}
% \zeta:=\gamma_{n+1}(a)
% & \in \Gamma_{n_{0}}\cap B(\gamma_{n+1}(t),2(1+\ve\beta)cM^{-N-1}) \notag \\
%&  \subseteq B(\gamma_{n+1}(t),(1+\ve\beta)M^{-1}\diam Q) \notag \\
%&  =  B\ps{\gamma_{n+1}(t),(1+\ve\beta)\frac{\ve\beta}{8} \diam Q} \notag \\
%&  \subseteq B\ps{\gamma_{n+1}(t),\frac{\ve\beta}{4}\diam Q}  \subseteq B\ps{\gamma_{n+1}(t),\frac{C\beta}{16}\diam Q}
%\end{align}
% %
 for $\ve<\frac{C}{4}$, and so
 \begin{equation}
 B\ps{\zeta,\frac{C\beta}{16}\diam Q}\subseteq B\ps{\gamma_{n+1}(t),\frac{C\beta}{8}\diam Q}\subseteq Q\backslash (\gamma_{n+1}(I))_{\frac{\beta}{128}\diam Q}.
 \label{e:b/128}
 \end{equation}
Thus, since $\Gamma_{n_{0}}$ is connected and $\diam \Gamma_{n_{0}}>\diam Q_{0}>\frac{C\beta}{16}\diam Q$, we know there is a curve $\rho\subseteq \Gamma_{n_{0}}\cap B(\zeta,\frac{C\beta}{16}\diam Q)$ connecting $\zeta$ to $B(\zeta,\frac{C\beta}{16}\diam Q)^{c}$, and hence has diameter at least $\frac{C\beta}{16}\diam Q$. Hence,
 \[\cH^{1}_{\infty}(\rho) \geq \diam \rho \geq \frac{C\beta}{16}\diam Q.\]
 Moreover,
\[
 \cH^{1}_{\infty}(\gamma(I))
 \geq  \diam \gamma(I) 
 \geq |\gamma(a_{I})-\gamma(b_{I})|
 \stackrel{\eqn{allthesame}}{=} |\gamma_{n}(a_{I})-\gamma_{n}(b_{I})|
  \stackrel{\eqn{ln}}{\geq} (1-3\ve\beta)\diam Q.
\]
Hence, since any cube in $\cL^{1}(Q)$ intersecting $\rho$ has diameter at most $\frac{\ve\beta}{4}\diam Q<\frac{\beta}{128}$ by \eqn{zeta-gamma}, they are disjoint from those intersecting $\gamma(I)$ by \eqn{b/128} if we choose $\ve<\frac{1}{128}$ (since if they intersect $\gamma(I)$, they also intersect $\gamma_{n+1}(I)$ by the definition of $\gamma_{n+1}$). Thus, we have
 \[\cH^{1}_{\infty}(Q)\geq \frac{C\beta}{16}\diam Q+(1-3\ve\beta)\diam Q\geq \ps{1+\frac{C\beta}{32}}\diam Q\]
 for $\ve<\frac{C}{96}$. Hence, by picking $K=\frac{C}{32}$, we see that $Q\in \Delta_{Bad}$, which finishes the proof of \Lemma{good-or-bad}

\subsection{Geometric martingales and the proof of \Lemma{bad}}
\label{s:lemma-bad}

For $Q\in \Delta$, define $k(Q)$ to be the number of cubes in $\Delta_{Bad}$ that properly contain $Q$, and set 
\[\Delta_{Bad,j}=\{Q\in \Delta_{Bad}: k(Q)=j\},\]
\[Bad_{j}(Q)=\{R\subseteq Q:k(R)=k(Q)+j\},\]
\[G(Q)=(\Gamma_{n_{0}}\cap Q)\backslash \bigcup_{R\in Bad_{1}(Q)}R.\]

We will soon define, for each $Q\in\Delta_{bad}$, a nonnegative weight function $w_{Q}:\Gamma_{n_{0}}\rightarrow [0,\infty)$ $\cH^{1}|_{\Gamma_{n_{0}}}$-a.e. in a martingale fashion by defining it as a limit of a sequence $w_{Q}^{j}$. Each $w_{Q}^{j}$ will be constant on various subsets of $\Gamma_{n_{0}}$ that partition $\Gamma_{0}$. We will actually  decide the value of $w_{Q}^{j}$ on an element $A$ of the partition, say, by declaring the value of
\[w_{Q}^{j}(A):=\int_{\Gamma_{n_{0}}\cap A }w_{Q}^{j}d\cH^{1}.\]
Then we will define $w_{Q}^{j+1}$ to be constant on sets in a partition subordinate to the previous partition so that, on sets $A$ in the $j$th partition, $w_{Q}^{j+1}(A)=w_{Q}^{j}(A)$, and so forth. We do this in such a way that we disseminate the mass of the weight function $w_{Q}$ so that $w_{Q}$ is supported in $Q$, has integral $\diam Q$, and so that $w_{Q}(x)\leq \frac{1}{(1+K\beta)^{k(x)-k(Q)}}$, where $k(x)$ is the total number of bad cubes containing $x$.  By geometric series, this will mean that $\sum_{Q\in \Delta_{Bad}}w_{Q}\one_{Q}$ is a bounded function, so that its total integral is at most a constant times $\cH^{1}(\Gamma_{0})$. However, the integral of each of these functions $w_{Q}$ is $\diam Q$, and so the integral is also equal to $\sum_{Q\in \Delta_{Bad}}\diam Q$, which gives us \eqn{bad}. This method appears in \cite{Schul-TSP}. Now we proceed with the proof.

 First set
\begin{equation}
w_{Q}^{0}(Q)=\diam Q,  \;\;\; w_{Q}^{0}|_{Q^{c}}\equiv 0
\label{e:wQQ}
\end{equation}
and construct $w_{Q}^{j+1}$ from $w_{Q}^{j}$ as follows:
\begin{enumerate}
\item If $R\in Bad_{j}(Q)$ for some $j$, and $S\in Bad_{1}(R)$, set $w_{Q}^{j+1}$ to be constant in $S$ so that 
\begin{equation}
w_{Q}^{j+1}(S)=w_{Q}^{j}(R)\frac{\diam S}{\sum_{T\in Bad_{1}(R)} \diam T+\cH^{1}(G(R))}.
\label{e:wQS}
\end{equation}
\item  Set $w_{Q}^{j+1}$ to be constant in $G(R)$ so that 
\begin{equation}
w_{Q}^{j+1}(G(R))=w_{Q}^{j}(R)-\sum_{S\in Bad_{1}(R)}w_{Q}^{j+1}(S).
\label{e:wQG}
\end{equation}
\item For points $x$ not in in any $R\in Bad_{j}(Q)$, set $w_{Q}^{j+1}(x)=w_{Q}^{j}(x)$. 
\end{enumerate}
Like a martingale, we have by our construction that, if $R\in Bad_{j}(Q)$, then $w_{Q}^{i}(R)=w_{Q}^{j}(R)$ for all $i\geq j$, and in particular, $w_{Q}^{j}(Q)=\diam Q$ for all $j\geq 0$. \\

We will need the following inequality:
\begin{equation}
\sum_{T\in Bad_{1}(R)} \diam T+\cH^{1}(G(R))
\geq \cH^{1}_{\infty}(R\cap\Gamma_{n_{0}})\geq (1+K\beta)\diam R.
\label{e:>delta}
\end{equation}
The first inequality comes from the fact that if $\delta>0$ and $A_{i}$ is a cover of $G(R)$ by sets so that $\sum\diam A_{i}<\cH^{1}(G(R))+\delta$, then $\{A_{i}\}\cup Bad_{1}(R)$ is a cover of $R$ (up to a set of $\cH^{1}$-measure zero by \Lemma{zero-boundary}), and so
\begin{align*}
\sum_{T\in Bad_{1}(R)}\diam T+\cH^{1}(G(R))+\delta & 
 >\sum\diam A_{i}+\sum_{T\in Bad_{1}(R)}\diam T \\
 & \geq \cH^{1}_{\infty}(R\cap \Gamma_{n_{0}})
 \end{align*}
which gives the first inequality in \eqn{>delta} by taking $\delta\rightarrow 0$. The last inequality in \eqn{>delta} is from the definition of $\Delta_{Bad}$. 

For $S\in Bad_{1}(R)$ and $R\in Bad_{j}(Q)$, by induction we have 
\begin{align}
\frac{w_{Q}^{j+1}(S)}{\diam S}
& \stackrel{\eqn{wQS}}{=} \frac{w_{Q}^{j}(R)}{\sum_{T\in Bad_{1}(R)} \diam T+\cH^{1}(G(R))} 
\stackrel{\eqn{>delta}}{\leq} \frac{w_{Q}^{j}(R)}{\diam R}\frac{1}{1+K\beta} \notag \\
& \leq \frac{w_{Q}^{0}(Q)}{\diam Q}\frac{1}{(1+K\beta)^{j+1}}
\stackrel{\eqn{wQQ}}{=}\frac{1}{(1+K\beta)^{j+1}}
\label{e:1+Kvebj}
\end{align}
Hence, since $w_{Q}^{j+1}$ is constant in $S$, for $x\in S\cap \Gamma_{n_{0}}$,
\begin{align}
w_{Q}^{j+1}(x)
&  \stackrel{\eqn{wQS}}{=}w_{Q}^{j}(R) \frac{\diam S}{\sum_{T\in Bad_{1}(R)}\diam T+\cH^{1}(G(R))}\frac{1}{\cH^{1}(S\cap \Gamma_{n_{0}})} \notag \\
& \stackrel{\eqn{>delta}}{ \leq} \frac{w_{Q}^{j}(R) }{\sum_{T\in Bad_{1}(R)}\diam T+\cH^{1}(G(R))}\frac{1}{1+K\beta} \notag \\
& \stackrel{\eqn{>delta}}{ \leq}  \frac{w_{Q}^{j}(R)}{\diam R}\frac{1}{(1+K\beta)^{2}} \stackrel{\eqn{1+Kvebj}}{\leq} \frac{w_{Q}^{0}(Q)}{\diam Q}\frac{1}{(1+K\beta)^{j+2}} 
  =\frac{1}{(1+K\beta)^{j+2}}.
\label{e:wonR}
\end{align}
Moreover, if $x\in G(R)$, 

\begin{align}
w_{Q}^{j+1}(x)
&=\frac{w_{Q}^{j+1}(G(R))}{\cH^{1}(G(R))}
 \stackrel{\eqn{wQG}}{=} \frac{w_{Q}^{j}(R)-\sum_{S\in Bad_{1}(R)}w_{Q}^{j+1}(S)}{\cH^{1}(G(R))} \notag \\
& \stackrel{\eqn{wQS}}{=}\frac{w_{Q}^{j}(R)}{\cH^{1}(G(R))}\ps{1-\sum_{S\in Bad_{1}(R)}\frac{\diam S}{\sum_{T\in Bad_{1}(R)}\diam T+\cH^{1}(G(R))}} \notag \\
& =\frac{w_{Q}^{j}(R)}{\cH^{1}(G(R))}\frac{\cH^{1}(G(R))}{\sum_{T\in Bad_{1}(R)}\diam T+\cH^{1}(G(R))} \notag \\
 & =  \frac{w_{Q}^{j}(R)}{\sum_{T\in Bad_{1}(R)}\diam T+\cH^{1}(G(R))} 
  \stackrel{\eqn{>delta}}{<}\frac{w_{Q}^{j}(R)}{\diam R}\frac{1}{1+K\beta} \\
  & 
\stackrel{\eqn{1+Kvebj}}{\leq}\frac{1}{(1+K\beta)^{j+1}} 
\label{e:wonG}
\end{align}

Since $\Delta_{Bad}\subseteq \bigcup_{j=0}^{n_{0}}\Delta_{j}$, and $\cH^{1}(\bigcup_{Q\in \Delta}\d Q)=0$, almost every point $x\in Q_{0}\cap \Gamma_{n_{0}}$ is contained in at most finitely many cubes in $\Delta_{Bad}$, and hence the value of $w_{Q}^{j+1}(x)$ changes only finitely many times in $j$, thus the limit $w_{Q}=\lim_{j}w_{Q}^{j}$ is well defined almost everywhere. For $x\in Q\cap \Gamma_{n_{0}}$, set $k(x)=k(R)$ where $R\subseteq Q$ is the smallest cube in $\Delta_{Bad}$ containing $x$. Then \eqn{wonR} and \eqn{wonG} imply
\[w_{Q}(x)\leq \frac{1}{(1+K\beta)^{k(x)-k(Q)}}\]
and so
\[\sum_{x\in Q\in \Delta_{Bad}}w_{Q}(x) \leq \sum_{j=0}^{k(x)}\frac{1}{(1+K\beta)^{j}}
\leq \sum_{j=0}^{\infty}\frac{1}{(1+K\beta)^{j}} 
= \frac{1+K\beta}{K\beta}\leq \frac{2}{K\beta}\]
since $K\beta< 1$. Hence,
\begin{align*}
\sum_{Q\in \Delta_{Bad}}\diam Q
& =\sum_{Q\in \Delta_{Bad}}\int_{Q}w_{Q}(x)d\cH^{1}(x)
 =\int_{\Gamma_{n_{0}}}\ps{\sum_{x\in Q\in \Delta_{Bad}} w_{Q}(x)}d\cH^{1}(x)\\
&  \leq \frac{2}{K\beta}\cH^{1}(\Gamma_{n_{0}})
\end{align*}
which finishes the proof of \Lemma{bad}.

\section{Antenna-like sets}
\label{s:antenna}

This section is devoted to the proof of  \Theorem{antenna-like}.

It is easy to verify using the definitions that being antenna-like is a quasisymmetric invariant quantitatively, so by \Theorem{main}, it suffices to verify that, if $ X$ is $c$-antenna-like, then any ball $B(x,r)$ with $x\in  X$ and $0<r<\frac{\diam X}{2}$ has $\beta'(x,r)>\frac{c}{7}$.

Fix such a ball, so there is a homeomorphism $h:\bigcup_{i=1}^{3}[0,e_{i}]\rightarrow X\cap B(x,r)$ so that 
\begin{equation}
\dist (h(e_{i}),h([0,e_{j}]\cup [0,e_{k}]))\geq cr
\label{e:hy}
\end{equation}
for all permutations $(i,j,k)$ of $(1,2,3)$ (see Figure \ref{f:antenna-beta}).

Let $s:[0,1]\rightarrow B(x,r)$ satisfy
\[\ell(s|_{[0,1]})-|s(0)-s(1)|+\sup_{z\in X\cap B(x,r)}\dist(z,s([0,1]))<2\beta'(x,r)|s_{0}-s_{1}|=:\beta.\]
Set $x_{i}=h(e_{i})$ for $i=1,2,3$  and let
\[
t_{1}=\inf s^{-1}\ps{\bigcup_{i=1}^{3}B(x_{i},\beta)}.
\]
This always exists since $X\cap B(x,r)\subseteq (s([0,1]))_{\beta}$. Without loss of generality, assume $s(t_{1})\in B(x_{1},\beta).$ Similarly, let
\begin{equation}
t_{2}=\inf s^{-1}\ps{\bigcup_{i=2}^{3}B(x_{i},\beta)}
\label{e:t_2}
\end{equation}
and again, without loss of generality, assume $s(t_{2})\in B(x_{2},\beta)$.

%\begin{figure}[h]
%\begin{picture}(100,180)(-90,0)
%\put(0,0){\includegraphics[width=200pt]{antenna-beta.pdf}}
%%\putgrid
%\put(37,77){$s(\zeta_{2})$}
%\put(60,60){$z$}
%\put(13,35){$x_{1}$}
%\put(2,17){$s(t_{1})$}
%\put(175,33){$s(t_{2})$}
%\put(75,30){$s(\zeta_{1})$}
%\put(75,177){$x_{3}$}
%\put(175,65){$x_{2}$}
%\put(80,110){${\color{gray} h(Y)\subseteq X\cap B(x,r)}$}
%\end{picture}
%\caption{}
%\label{f:antenna-beta}
%\end{figure}

\begin{figure}[h]
\begin{picture}(100,180)(45,0)
\put(0,0){\includegraphics[width=200pt]{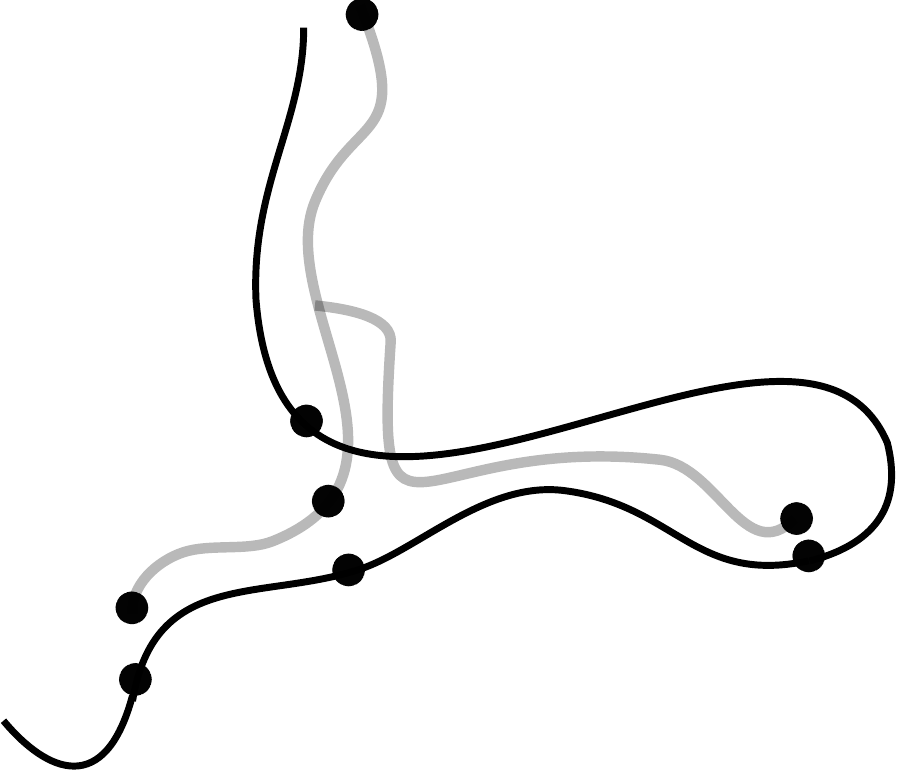}}
%\putgrid
\put(37,77){$s(\zeta_{2})$}
\put(60,60){$z$}
\put(13,35){$x_{1}$}
\put(2,17){$s(t_{1})$}
\put(175,33){$s(t_{2})$}
\put(75,30){$s(\zeta_{1})$}
\put(75,177){$x_{3}$}
\put(175,65){$x_{2}$}
\put(80,110){${\color{gray} h(Y)\subseteq X\cap B(x,r)}$}
\end{picture}
\caption{}
\label{f:antenna-beta}
\end{figure}

Note that $h([0,e_{1}]\cup [0,e_{3}])$ is a path connecting $x_{1}$ to $x_{3}$, where the latter point is not contained in $(s([t_{1},t_{2}]))_{\beta}$ by our choices of $t_{1}$ and $t_{2}$, although the latter point is; otherwise, there would be $t\in [t_{1},t_{2}]$ such that $s(t)\in B(x_{3},\beta)$, contradicting the minimality of $t_{2}$. Since $h([0,e_{1}]\cup [0,e_{3}])$ is connected and $(s([t_{1},t_{2}]))_{\beta}$ contains $x_{1}$ but not $x_{3}$, we can pick a point $z\in h([0,e_{1}]\cup [0,e_{3}])$ so that $\dist (z,s([t_{1},t_{2}]))=\beta$ . Pick $\zeta_{1}\in [t_{1},t_{2}]$ and $\zeta_{2}\in (t_{2},1]$ so that
\begin{equation} 
|s(\zeta_{1})-z|=\dist (z,s([t_{1},t_{2}]))=\beta  \;\;\; \mbox{ and } \;\;\; |s(\zeta_{2})-z|<\beta.
\label{e:zeta12}
\end{equation}
Then by \Lemma{subarc},
\begin{align*}
2\beta'(x,r)|s_{0}-s_{1}|
& >\ell(s|_{[0,1]})-|s(0)-s(1)| \geq \ell(s|_{[\zeta_{1},\zeta_{2}]})-|s(\zeta_{1})-s(\zeta_{2})|\\
& \geq  \ell(s|_{[\zeta_{1},t_{2}]}) + \ell(s|_{[t_{2}, \zeta_{2}]})   -|s(\zeta_{1})-z|-|z-s(\zeta_{2})|\\
& \hspace{-3pt}\stackrel{\eqn{zeta12}}{>} |s(\zeta_{1})-s(t_{2})|+|s(t_{2})-s(\zeta_{2})|-\beta-\beta\\
& \geq |z-x_{2}|-|s(\zeta_{1})-z|-|x_{2}-s(t_{2})| \\
& \hspace{13pt}+|x_{2}-z|-|s(t_{2})-x|-|s(\zeta_{2})-z|-2\beta\\
& \hspace{-13pt}\stackrel{\eqn{hy},\eqn{zeta12}}{ \geq} cr-\beta-\beta+cr-\beta-\beta-2\beta \\
& =2cr-6\beta\geq c|s(0)-s(1)|-12\beta(x,r)|s(0)-s(1)|
\end{align*}
which yields $\beta'(x,r)\geq \frac{c}{7}$ and completes the proof of  \Theorem{antenna-like}

\section{Comparison of the $\beta$-numbers}
\label{s:betas}

For quantities $A$ and $B$, we will write $A\lec B$ if there is a universal constant $C$ so that $A\leq CB$, and $A\sim B$ if $A\lec B\lec A$.

\begin{lemma}
Let $X\subseteq \ell^{\infty}$ be a compact connected set, $x\in  X$, and $0<r<\frac{\diam X}{2}$. Then
\begin{equation}
\beta'(x,r) \leq \bhat(x,r)\lec \beta'(x,r)^{\frac{1}{2}}.
\end{equation}
\label{l:beta'-bhat}
\end{lemma}

\begin{proof}
The first inequality follows trivially from the definitions, since each sequence $y_{0},...,y_{n}\in  X$ induces a finite polygonal Lipschitz path $s$ in $\ell^{\infty}$ for which
\[\ell(s)-|s(0)-s(1)|=\sum_{i=0}^{n-1}|y_{i}-y_{i+1}|-|y_{0}-y_{n}|.\]

For the opposite inequality, let $s:[0,1]\rightarrow\ell^{\infty}$ be such that
\begin{equation}
\frac{\ell(s)-|s(0)-s(1)| + \sup_{z\in B(x,r)\cap X}\dist(z,s([0,1]))}{|s(0)-s(1)|}\leq 2\beta'(x,r)=:\beta.
\label{e:infs}
\end{equation}
Let
\[A=s^{-1}\ps{(s([0,1]))_{2\beta|s_{0}-s_{1}|}}\]
which is a relatively open subset of $[0,1]$. Let $a=\inf A$  and define $a=t_{0}<t_{1}<\cdots<t_{n}\leq 1$ inductively by setting
\[t_{i+1}=\inf\{t\in A\cap (t_{i},b]: \dist (s(t),s([t_{0},t_{i}]))>\beta^{\frac{1}{2}}|s(0)-s(1)|\}.\]
We claim that 
\begin{equation}
n\sim \beta^{-\frac{1}{2}}|s(0)-s(1)|.
\label{e:nb1/2}
\end{equation}
To see this, note that since $|s(t_{i})-s(t_{i+1})|\geq \beta^{\frac{1}{2}}|s(0)-s(1)|$ , the sets $B(s(t_{i}), \frac{\beta^{\frac{1}{2}}}{2}|s(0)-s(1)|)$ are disjoint, so that
\[n\frac{\beta^{\frac{1}{2}}}{2}|s(0)-s(1)|\leq \ell(s)\leq (1+\beta)|s(0)-s(1)|\leq 2|s(0)-s(1)|\]
which gives $n\leq 4\beta^{-\frac{1}{2}}$. On the other hand, the balls $B(s(t_{i}), 2\beta^{\frac{1}{2}}|s(0)-s(1)|)$ cover $s([0,1])$, and so 
\begin{align*}
 |s(0)-s(1)|
& \leq \ell(s)
\leq \sum_{i=0}^{n} \diam B(s(t_{i}),2\beta^{\frac{1}{2}}|s(0)-s(1)|)\\
& \leq (n+1)4\beta^{\frac{1}{2}}|s(0)-s(1)|
\leq 8n\beta^{\frac{1}{2}}|s(0)-s(1)|
\end{align*}
which gives $n\geq (8\beta)^{-1}$, and this proves \eqn{nb1/2}.

By the definition of $A$, there are
\[y_{i}\in \cnj{B(s(t_{i}),2\beta|s(0)-s(1)|)}.\]
Then

\begin{multline*}
\sum_{i=0}^{n-1}|y_{i}-y_{i+1}|-|y_{0}-y_{1}|
\leq \sum_{i=0}^{n-1}|s(t_{i})-s(t_{i+1})|+4n\beta|s(0)-s(1)|-|s(t_{0})-s(t_{n})|\\
\stackrel{\eqn{nb1/2}}{\leq} \ell(s|_{[t_{0},t_{n}]})-|s(t_{0})-s(t_{n})|+C\beta^{\frac{1}{2}}|s(0)-s(1)|\\
\stackrel{\eqn{infs}}{\leq} \beta|s_{0}-s_{1}|+C\beta^{\frac{1}{2}}|s(0)-s(1)|\lec \beta^{\frac{1}{2}}|s(0)-s(1)|.
\end{multline*}
{\bf Claim:} $|s(0)-s(1)|\lec |s(t_{0})-s(t_{n})|$. 

Since $\diam$ is connected and $r<\frac{\diam X}{2}$, there is a path connecting $x$ to $B(x,r)^{c}$, which naturally must be of diameter at least $r$, hence
\begin{align*}
|s(0)-s(1)| & \leq  2 r\leq 2(\ell(s|_{[t_{0},t_{n}]})-4\beta|s_{0}-s_{1}|) \\ 
& \leq 2|s(t_{0})-s(t_{n})|+C\beta^{\frac{1}{2}}|s(0)-s(1)|,\end{align*}
which, if $\beta^{\frac{1}{2}}$ is small enough, this implies
\[|s(0)-s(1)|\leq 4|s(t_{0})-s(t_{n})|=4|y_{0}-y_{n}|\]
so that the above estimates imply
\begin{equation} \sum_{i}|y_{i}-y_{i+1}|-|y_{0}-y_{n}|\lec  \beta^{\frac{1}{2}}|s(0)-s(1)| \leq 4\beta^{\frac{1}{2}}|y_{0}-y_{n}|
\label{e:yb1/21}
\end{equation}
Moreover, 
\begin{align}
\diam \cap B(x,r) 
& \subseteq (s([0,1]))_{\beta|s(0)-s(1)|}
\subseteq \bigcup_{i} B(s(t_{i}),(2\beta^{\frac{1}{2}} +\beta)|s(0)-s(1)|) \notag \\
& \subseteq \bigcup_{i}B(y_{i},(2\beta^{\frac{1}{2}} +\beta+2\beta)|s(0)-s(1)|) \notag \\
& \subseteq \bigcup_{i}B(y_{i},5\beta^{\frac{1}{2}}|s(0)-s(1)|) 
\subseteq \bigcup_{i}B(y_{i},20\beta^{\frac{1}{2}}|y_{0}-y_{n}|)
\label{e:yb1/22}
\end{align}
Thus \eqn{yb1/21} and \eqn{yb1/22} imply $\bhat(x,r)\leq 20\beta^{\frac{1}{2}}= 20\sqrt{2}\beta'(x,r)^{\frac{1}{2}}$.

\end{proof}

\begin{proposition}
If $ X$ is a compact connected subset of some Hilbert space, then
\[\beta''(x,r)\leq \beta(x,r)\lec \beta''(x,r) \mbox{ for }x\in \Gamma \mbox{ and }r<\frac{\diam  X}{2}\]
where
\[\beta''(x,r) =\inf_{s} \ps{\ps{\frac{\ell(s)-|s(0)-s(1)|}{|s(0)-s(1)|}}^{\frac{1}{2}}  + \frac{\sup_{z\in B(x,r)\cap X]}\dist (z,s([0,1]))}{|s(0)-s(1)|}}.\]
In particular,
\begin{equation}
\beta'(x,r)\leq \beta(x,r)\lec \beta'(x,r)^{\frac{1}{2}}.
\label{e:bb''-prop}
\end{equation}
\label{p:bb''}
\end{proposition}

Note that \eqn{bb''-prop} is tight in the sense that if $ X\subseteq \bC$, $0\in  X$, and $B(0,1)\cap \Gamma=[-1,1]\cup [0,i\ve]$, then by \Theorem{antenna-like} and \eqn{bb''-prop}, for all $\ve>0$,
\[ \beta(0,1)\leq \ve \leq 7\beta'(0,1)\leq 7\beta(0,1)\leq 7\ve.\]
However, if $ X\cap B(x,r)=[-1,0]\cup [0,e^{i\ve}]$, then for all $\ve>0$, again by \eqn{bb''-prop} (and estimating $\beta''(0,1)$ by letting $s$ be the path traversing the segments $[-1,0]\cup [0,e^{i\ve}]$),
\[\beta(0,1)^{2}\sim \ve^{2}\gec \beta'(0,1)\gec \beta(0,1)^{2}.\]

\begin{proof}
For the first inequality, simply let $s:[0,1]\rightarrow \cH$ be the line segment spanning $L\cap B(x,r)$ where $L$ is some line passing through $B(x,\frac{r}{2})$. Then $\ell(s)=\cH^{1}(L\cap B(x,r))\geq r$ and hence
\[
\beta''(x,r)
\leq  \frac{\sup_{z\in B(x,r)\cap X}\dist (z,s([0,1]))}{|s(0)-s(1)|}
\leq  \frac{\sup_{z\in B(x,r)\cap X}\dist (z,L)}{r}.\]
Since $x\in  X$, the range of admissible lines in the infimum in \eqn{euclidean-beta} can be taken to be lines intersecting $B(x,\frac{r}{2})$. Using this fact and infimizing the above inequality over all such lines proves the first inequality in \eqn{bb''-prop}.

For the opposite inequality, let $s$ satisfy 
\[\ps{\frac{\ell(s)-|s(0)-s(1)|}{|s(0)-s(1)|}}^{\frac{1}{2}}  + \frac{\sup_{z\in B(x,r)\cap X}\dist (z,s([0,1]))}{|s(0)-s(1)|} \leq 2\beta''(B(x,r))=:\beta.\]
Let
\[\beta(s):=\sup_{t\in [0,1]}\dist (s(t),[s(0),s(1)]).\]
Then by the Pythagorean theorem, there is $c>0$ so that 
\[(1+c\beta(s)^{2})|s(0)-s(1)|\leq \ell(s)\leq (1+\beta^{2})|s(0)-s(1)|\]
so that $\beta(s)\leq c^{-1}\beta,$. Hence, if $L$ is the line passing through $s(0)$ and $s(1)$,
\begin{multline*}
\beta(x,r)
\leq \sup_{z\in B(x,r)\cap X}\dist (z,L)
\leq \sup_{z\in B(x,r)\cap X}\dist (z,[s(0),s(1)])\\
\leq \beta(s)+\sup_{z\in B(x,r)\cap X}\dist (z,s([0,1]))
\leq c^{-1}\beta+\beta
\lec \beta
\end{multline*}

\end{proof}

\bibliographystyle{amsplain}
%\bibliography{wiggly-reference}
%\end{document}

\def\cprime{$'$}
\providecommand{\bysame}{\leavevmode\hbox to3em{\hrulefill}\thinspace}
\providecommand{\MR}{\relax\ifhmode\unskip\space\fi MR }
% \MRhref is called by the amsart/book/proc definition of \MR.
\providecommand{\MRhref}[2]{%
  \href{http://www.ams.org/mathscinet-getitem?mr=#1}{#2}
}
\providecommand{\href}[2]{#2}

\end{document}